\documentclass[a4paper,11pt,reqno]{amsart}
\usepackage[english]{babel}
\usepackage[latin1]{inputenc}
\usepackage{cmbright}
\usepackage[T1]{fontenc}
\usepackage{amsmath}
\usepackage{amsfonts}
\usepackage{amssymb}
\usepackage{amsthm,epic,tikz}
\usepackage{mathrsfs}
\usepackage{enumitem}
\usepackage{graphicx}
\usepackage{footnote,color}
\usepackage{hyperref}
\hypersetup{colorlinks=true,linkcolor=blue,citecolor=red,filecolor=BrickRed,       urlcolor=darkgreen}

\renewcommand{\geq}{\geqslant}
\renewcommand{\leq}{\leqslant}
\newtheorem{thm}{Theorem}

\newtheorem{rem}[thm]{Remark}
\newtheorem{cor}[thm]{Corollary}

\newtheorem{lem}[thm]{Lemma}
\definecolor{redd}{rgb}{0.9,0,0}

\let\o=\omega

\usepackage{color}
\definecolor{darkgreen}{rgb}{0,0.4,0}
\definecolor{MyDarkBlue}{rgb}{0,0.08,0.50}
\definecolor{BrickRed}{rgb}{0.65,0.08,0}

\title[A human proof of Gessel's lattice path conjecture]{A human proof of Gessel's lattice path conjecture}

\author{A.\ Bostan}
\thanks{A.\ Bostan: INRIA Saclay \^Ile-de-France, B\^atiment Alan Turing, 1 rue Honor\'e d'Estienne d'Orves, 91120 Palaiseau, France}

\author{I.\ Kurkova}
\thanks{I.\ Kurkova: Laboratoire de Probabilit\'es et Mod\`eles Al\'eatoires, Universit\'e Pierre et Marie Curie, 4 Place Jussieu, 75252 Paris Cedex 05, France}

\author{K.\ Raschel}
\thanks{K.\ Raschel: CNRS \& F\'ed\'eration de recherche Denis Poisson \& Laboratoire de Math\'ematiques et Physique Th\'eorique, Universit\'e de Tours, Parc de Grandmont, 37200 Tours, France}

\thanks{Emails: \url{Alin.Bostan@inria.fr}, \url{Irina.Kourkova@upmc.fr}, and \url{Kilian.Raschel@lmpt.univ-tours.fr}}


\date{\today}

\begin{document}

\begin{abstract}
Gessel walks are lattice paths confined to the quarter plane that start at the origin and consist of unit steps going either West, East, South-West or North-East.
In 2001, Ira Gessel conjectured a nice closed-form expression for the number of Gessel walks ending at the origin. In 2008, Kauers, Koutschan and Zeilberger gave a computer-aided proof of this conjecture. The same year, Bostan and Kauers showed, again using computer algebra tools, that the complete generating function of Gessel walks is algebraic.
In this article we propose the first ``human proofs'' of these results. They are derived from a new expression for the generating function of Gessel walks in terms of Weierstrass zeta functions.
\end{abstract}

\maketitle

\section{Introduction}
\label{sec:Introduction}

\subsection*{Main results}
Gessel walks are lattice paths confined to the quarter plane ${\bf N}^{2}=\{0,1,2,\ldots\} \times \{0,1,2,\ldots\}$, that start at the origin $(0,0)$ 
and move by unit steps in one of the following directions: West, East, South-West and North-East,
  see Figure~\ref{Steps_Gessel}.
Gessel excursions are those Gessel walks that return to the origin.
\unitlength=0.6cm
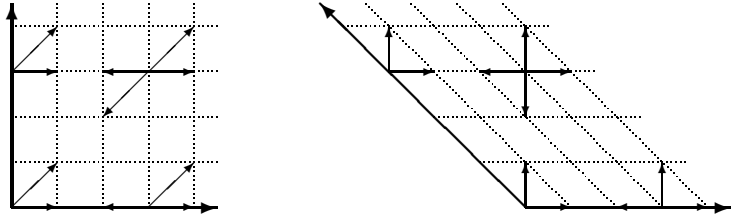
\begin{figure}[t]
  \begin{center}
\begin{tabular}{cccc}
    \hspace{-0.9cm}
        \begin{picture}(5,5.5)
    \thicklines
    \put(1,1){{\vector(1,0){4.5}}}
    \put(1,1){\vector(0,1){4.5}}
    \thinlines
    \put(4,4){\vector(1,1){1}}
    \put(4,4){\vector(-1,-1){1}}
    \put(4,4){\vector(1,0){1}}
    \put(4,4){\vector(-1,0){1}}
    \put(4,1){\vector(1,0){1}}
    \put(4,1){\vector(-1,0){1}}
    \put(4,1){\vector(1,1){1}}
    \put(1,4){\vector(1,0){1}}
    \put(1,4){\vector(1,1){1}}
    \put(1,1){\vector(1,0){1}}
    \put(1,1){\vector(1,1){1}}
    \linethickness{0.1mm}
    \put(1,2){\dottedline{0.1}(0,0)(4.5,0)}
    \put(1,3){\dottedline{0.1}(0,0)(4.5,0)}
    \put(1,4){\dottedline{0.1}(0,0)(4.5,0)}
    \put(1,5){\dottedline{0.1}(0,0)(4.5,0)}
    \put(2,1){\dottedline{0.1}(0,0)(0,4.5)}
    \put(3,1){\dottedline{0.1}(0,0)(0,4.5)}
    \put(4,1){\dottedline{0.1}(0,0)(0,4.5)}
    \put(5,1){\dottedline{0.1}(0,0)(0,4.5)}
    \end{picture}
    \hspace{35mm}
         \begin{picture}(5,5.5)
    \thicklines
    \put(1,1){{\vector(1,0){4.5}}}
    \put(1,1){\vector(-1,1){4.5}}
    \thinlines
    \put(1,4){\vector(0,1){1}}
    \put(1,4){\vector(0,-1){1}}
    \put(1,4){\vector(1,0){1}}
    \put(1,4){\vector(-1,0){1}}
    \put(4,1){\vector(-1,0){1}}
    \put(4,1){\vector(0,1){1}}
    \put(4,1){\vector(1,0){1}}
    \put(-2,4){\vector(1,0){1}}
    \put(-2,4){\vector(0,1){1}}
    \put(1,1){\vector(1,0){1}}
    \put(1,1){\vector(0,1){1}}
    \linethickness{0.1mm}
    \put(0,2){\dottedline{0.1}(0,0)(4.5,0)}
    \put(-1,3){\dottedline{0.1}(0,0)(4.5,0)}
    \put(-2,4){\dottedline{0.1}(0,0)(4.5,0)}
    \put(-3,5){\dottedline{0.1}(0,0)(4.5,0)}
    \put(2,1){\dottedline{0.1}(0,0)(-4.5,4.5)}
    \put(3,1){\dottedline{0.1}(0,0)(-4.5,4.5)}
    \put(4,1){\dottedline{0.1}(0,0)(-4.5,4.5)}
    \put(5,1){\dottedline{0.1}(0,0)(-4.5,4.5)}
    \end{picture}
    \end{tabular}
  \end{center}
  \vspace{-4mm}
\caption{On the left: allowed steps for Gessel walks. Note that on the boundary of ${\bf N}^{2}$, the steps that would take the walks out of ${\bf N}^{2}$ are discarded.  On the right: an equivalent formulation of Gessel walks as the simple walks evolving in the cone with opening $135^\circ$.}
\label{Steps_Gessel}
\end{figure}
For $(i,j) \in {\bf N}^{2}$ and $n\geq 0$, let $q(i,j;n)$ be the number of 
Gessel walks of length~$n$ ending at the point $(i,j)$. 
Gessel walks have been puzzling the combinatorics community since 2001, when Ira Gessel conjectured: 
\begin{enumerate}[label={\rm (\Alph{*})},ref={\rm (\Alph{*})}]
     \item\label{enumi:problem_A} For all $n\geq 0$, the following closed-form expression holds for the number of Gessel excursions of even length $2n$
\begin{equation}
\label{eq:Gessel's_conjecture}
     q(0,0;2n)=16^{n} \frac{(5/6)_{n} (1/2)_{n}}{(2)_{n} (5/3)_{n}},
\end{equation}
where $(a)_{n}=a(a+1)\cdots (a+n-1)$ denotes the Pochhammer symbol.
\end{enumerate}
Note that obviously there are no Gessel excursions of odd length, that is $q(0,0;2n+1)=0$ for all $n\geq 0$. 
In 2008, Kauers, Koutschan and Zeilberger~\cite{KKZ} provided a computer-aided proof of this conjecture. 

A second intriguing question was to decide whether or not:
\begin{enumerate}[label={\rm (\Alph{*})},ref={\rm (\Alph{*})}]
\setcounter{enumi}{1}
     \item\label{enumi:problem_B}
      Is the complete generating function (GF) of Gessel walks
     \begin{equation}
     \label{eq:def_generating_function}
          Q(x,y;z) = \sum_{i,j,n\geq 0}q(i,j;n) x^i y^j z^n
     \end{equation} D-finite\footnote{The function $Q(x,y;z)$ is called D-finite if the vector space over ${\bf Q}(x,y,z)$---the field of rational functions in the three variables $x,y,z$---spanned by the set of all partial derivatives of $Q(x,y;z)$ is finite-dimensional, see for instance~\cite{Lipshitz}.}, or even algebraic (i.e., root of a non-zero polynomial in ${\bf Q}(x,y,z)[T]$)?

\end{enumerate}
The answer to this question --namely, the (initially unexpected) algebraicity of $Q(x,y;z)$-- was finally obtained by Bostan and Kauers~\cite{BK2}, using computer algebra techniques.

\begin{quote}
    \em In summary, the only existing proofs for Problems~\ref{enumi:problem_A} and~\ref{enumi:problem_B} used heavy computer calculations in a crucial way. In this article we obtain a new explicit expression for $Q(x,y;z)$, from which we derive the first ``human proofs'' of \ref{enumi:problem_A} and \ref{enumi:problem_B}.
\end{quote}

\subsection*{Context of Gessel's conjecture} In 2001, the motivation for
considering Gessel's model of walks was twofold. First, by an obvious linear
transformation, Gessel's walk can be viewed as the simple walk (i.e., with
allowed steps to the West, East, South and North) constrained to lie in a cone
with angle $135^\circ$, see Figure \ref{Steps_Gessel}. It turns out that
before 2001, the simple walk was well studied in different cones.
P\'olya~\cite{Pol21} first considered the simple walk in the whole plane
(``drunkard's walk''), and remarked that the probability that a simple random
walk ever returns to the origin is equal to 1. This is a consequence of the
fact that there are exactly $\binom{2n}{n}^2$ simple excursions of length $2n$
in the plane ${\bf Z}^{2}$. There also exist formul\ae\ for simple excursions
of length $2n$ evolving in other regions of ${\bf Z}^{2}$:
$\binom{2n+1}{n}C_n$ for the half plane~${\bf Z} \times {\bf N}$, and $C_n
C_{n+1}$ for the quarter plane~${\bf N}^{2}$, where $C_n =
\frac{1}{n+1}\binom{2n}{n}$ is the Catalan number~\cite{Arques}.
Gouyou-Beauchamps \cite{Gou86} found a similar formula $C_n C_{n+2} -
C_{n+1}^2$ for the number of simple excursions of length $2n$ in the cone with
angle $45^\circ$ (the first octant). It was thus natural to consider the cone
with angle $135^\circ$, and this is what Gessel did.

The second part of the motivation is that Gessel's model is a particular
instance of walks in the quarter plane. In 2001 there were already several
famous examples of such models: Kreweras' walk \cite{Kreweras,FL1,FL2,BM05}
(with allowed steps to the West, North-East and South) for which the GF
\eqref{eq:def_generating_function} is algebraic; Gouyou-Beauchamps's
walk~\cite{Gou86}; the simple walk~\cite{GKS}. Further, around 2000, walks in
the quarter plane were brought up to date, notably by Bousquet-M\'elou and
Petkov\v sek \cite{BMT1,BMT2}. Indeed, they were used to illustrate the
following phenomenon: although the numbers of walks satisfy a (multivariate)
linear recurrence with constant coefficients, their GF
\eqref{eq:def_generating_function} might be non-D-finite; see \cite{BMT2} for
the example of the knight walk.

\subsection*{Existing results in the literature} After 2001, many approaches
appeared for the treatment of walks in the quarter plane. Bousquet-M\'elou and
Mishna initiated a systematic study of such walks with small steps (this means
that the step set, i.e., the set of allowed steps for the walk, is a subset of
the set of the eight nearest neighbors). Mishna \cite{MM07,MM09} first
considered the case of step sets of cardinality three. She presented a
complete classification of the GF~\eqref{eq:def_generating_function} of these
walks with respect to the classes of algebraic, transcendental D-finite and
non-D-finite power series. Bousquet-M\'elou and Mishna \cite{BMM} then
explored all the $79$ small step sets\footnote{A priori, there are $2^8=256$
step sets, but the authors of \cite{BMM} showed that, after eliminating
trivial cases, and also those which can be reduced to walks in a half plane,
there remain $79$ inherently different models.}. They considered a functional
equation for the GF that counts walks in such a model leading to a
group\footnote{Historically, this group was introduced by Malyshev
\cite{MA,MAL,MALY} in the seventies. For details on this group we refer to
Section \ref{sec:mero}, in particular to equation \eqref{xieta}.} of
birational transformations of ${\bf C}^2$. In $23$ cases out of $79$ this
group turns out to be finite, and the corresponding functional equations were
solved in $22$ out of $23$ cases (the finiteness of the group being a crucial
feature in \cite{BMM}). The remaining case was precisely Gessel's. In 2008, a
method using computer algebra techniques was proposed by Kauers, Koutschan and
Zeilberger \cite{KZ,KKZ}. Kauers and Zeilberger \cite{KZ} first obtained a
computer-aided proof of the algebraicity of the GF counting Kreweras' walks. A
few months later, this approach was enhanced to cover Gessel's case, and the
conjecture (Problem \ref{enumi:problem_A}) was proved~\cite{KKZ}. At the same
time, Bostan and Kauers \cite{BK2} showed, again using heavy computer
calculations, that the complete GF counting Gessel walks is algebraic (Problem
\ref{enumi:problem_B}). Using the minimal polynomials obtained by Bostan and
Kauers, van Hoeij \cite[Appendix]{BK2} managed to obtain an explicit and
compact expression for the complete GF of Gessel walks.

Since the computerized proofs~\cite{KKZ,BK2}, several computer-free analyses of the GF of Gessel walks have been proposed~\cite{KRG,Ra,FR,Ayyer,PW,SP,KRnew}, but none of them solved
Gessel's conjecture (Problem~\ref{enumi:problem_A}), nor proved the
algebraicity of the complete GF (Problem \ref{enumi:problem_B}). We briefly
review the contributions of these works. Kurkova and Raschel~\cite{KRG}
obtained an explicit integral representation (a Cauchy integral) for
$Q(x,y;z)$. This was done by solving a boundary value problem, a method
inspired by the book~\cite{FIM}. It can be deduced from \cite{KRG} that the
generating function \eqref{eq:def_generating_function} is D-finite, since the
Cauchy integral of an algebraic function is D-finite~\cite{Picard,RoYo05}.
This partially solves Problem \ref{enumi:problem_B}. Nevertheless, the
representation of \cite{KRG} seems to be hardly accessible for further
analyses, such as for expressing the coefficients $q(i,j;n)$ in any
satisfactory manner, and in particular for providing a proof of Gessel's
conjecture. This approach has been generalized subsequently for all models of
walks with small steps in the quarter plane, see \cite{Ra}. In \cite{FR},
Fayolle and Raschel gave a proof of the algebraicity of the bivariate GF
(partially solving Problem~\ref{enumi:problem_B}), using probabilistic and
algebraic methods initiated in~\cite[Ch.~4]{FIM}: more specifically, they
proved that for any fixed value $z_0\in(0,1/4)$, the bivariate generating
function $Q(x,y;z_0)$ for Gessel walks is algebraic over $\mathbf R(x,y)$,
hence over $\mathbf Q(x,y)$. The same approach gives the nature of the
bivariate GF in all the other 22 models with finite group. It is not possible
to deduce Gessel's conjecture using the approach in~\cite{FR}, since it only
uses the structure of the solutions of the functional equation
\eqref{functional_equation} satisfied by the generating function
\eqref{eq:def_generating_function}, but does not give access to any explicit
expression. Ayyer~\cite{Ayyer} proposed a combinatorial approach inspired by
representation theory. He interpreted Gessel walks as words on certain
alphabets. He then reformulated $q(i,j;n)$ as numbers of words, and expressed
very particular numbers of Gessel walks. Petkov\v{s}ek and Wilf \cite{PW}
stated new conjectures, closely related to Gessel's. They found an expression
for Gessel's numbers in terms of determinants of matrices, by showing that the
numbers of walks are solution to an infinite system of equations. Ping
\cite{SP} introduced a probabilistic model for Gessel walks, and reduced the
computation of $q(i,j;n)$ to the computation of a certain probability. Using
then probabilistic methods (such as the reflection principle) he proved two
conjectures made by Petkov\v{s}ek and Wilf in~\cite{PW}. Very recently, using
the Mittag-Leffler theorem in a constructive way, Kurkova and
Raschel~\cite{KRnew} obtained new series expressions for the GFs of all models
of walks with small steps in the quarter plane, and worked out in detail the
case of Kreweras' walks. The present article is strongly influenced
by~\cite{KRnew} and can be seen as a natural prolongation of it.

\subsection*{Presentation of our method and organization of the article}

We fix $z\in (0, 1/4)$. To solve Problems
\ref{enumi:problem_A} and \ref{enumi:problem_B},
 we start from the GFs $Q(x,0;z)$ and $Q(0,y;z)$
 and from the functional equation (see e.g.~\cite[\S4.1]{BMM})
\begin{multline}
\label{functional_equation}
        K(x,y;z)Q(x,y;z)=\\K(x,0;z)Q(x,0;z)+K(0,y;z)Q(0,y;z)-K(0,0;z)
         Q(0,0;z)-x y, \qquad \forall \vert x\vert, \vert y\vert<1.
\end{multline}
Above, $K(x,y;z)$ is the kernel of the walk, given by
\begin{equation}
\label{def_kernel}
       K(x,y;z)=xyz\left(\sum_{(i,j)\in \mathfrak{G} } x^i y^j-1/z\right)= xyz(xy+x+1/x+1/(xy)-1/z),
\end{equation}
where $\mathfrak{G}=\{(1,1), (1,0), (-1,0), (-1,-1)\}$
denotes Gessel's step set (Figure \ref{Steps_Gessel}).

Rather than deriving an expression directly for the GFs $Q(x,0;z)$ and
$Q(0,y;z)$, we shall (equivalently) obtain expressions for $Q(x(\o),0;z)$ and
$Q(0,y(\o);z)$ for all $\o \in {\bf C}_\o$, where ${\bf C}_\o$ denotes the
complex plane, and where the functions $x(\o)$ and $y(\o)$ arise for reasons
that we now present. This idea of introducing the $\omega$-variable might
appear unnecessarily complicated; in fact it is very natural, in the sense
that many technical aspects of the reasonings will appear simple on the
complex plane ${\bf C}_\o$ (in particular, the group of the walk, the
continuation of the GFs, their regularity, their explicit expressions, etc.).

We shall see in Section \ref{sec:mero} that the elliptic curve defined by the
zeros of the kernel
\begin{equation}
\label{eq:def_T_z}
     {\bf T}_z=\{(x,y)\in({\bf C}\cup \{\infty\})^2 : K(x,y;z)=0\}
\end{equation}
is of genus $1$. This is a torus (constructed from two complex spheres
properly glued together), or, equivalently, a parallelogram $\o_1[0,1] +
\o_2[0,1]$ whose opposite edges are identified. It can be parametrized by
\begin{equation}
\label{eq:def_T_z_bis}
     {\bf T}_z = \{(x(\o),y(\o)) : \o\in {\bf C}/(\o_1{\bf Z}+\o_2{\bf Z})\},
\end{equation}
where the complex number $\o_1\in i{\bf R}$ and the real number $\o_2\in {\bf
R}$ are given in \eqref{expression_omega_1_2}, and the functions $x(\o),y(\o)$
are made explicit in \eqref{eq:uniformization_Gessel} in terms of the
Weierstrass elliptic function $\wp$ with periods $\o_1,\o_2$. By construction,
$K(x(\o),y(\o);z)$ is identically zero.
   
Equation \eqref{eq:uniformization_Gessel} and the periodicity of $\wp$ imply
that the functions $x(\o)$ and $y(\o)$ are {\it elliptic} functions on ${\bf
C}_\o$ with periods $\o_1,\o_2$. The complex plane ${\bf C}_\o$ can thus be
considered as the universal covering of ${\bf T}_z$ and can be viewed as the
union $\cup_{n,m \in {\bf Z}}\{\o_1[n,n+1)+ \o_2[m,m+1)\}$ of infinitely many
parallelograms, with the natural projection ${\bf C}_\o\to {\bf C}/(\o_1{\bf
Z}+\o_2{\bf Z})$ (see Figure \ref{Fig_intro}). Our first aim is to lift on it
the unknown functions $Q(x,0;z)$ and $Q(0,y;z)$, initially defined on their
respective unit disk.

The domain of $\o_1[0, 1)+ \o_2[0, 1)$ where $|x(\o)|<1$ is delimited by two
vertical lines (see Figure \ref{Fig_intro}). Due to the ellipticity of
$x(\o)$, the corresponding domain $\{\o \in {\bf C}_\o : |x(\o)|<1\}$ on the
universal covering ${\bf C}_\o$ consists of infinitely many vertical strips.
One of them, denoted by $\Delta_x$, belongs to the strip $i{\bf R}+ \o_2[0,
1)$; the other ones are its shifts by multiples of $\o_2>0$.

Next, we lift $Q(x,0;z)$ to the strip $\Delta_x \subset {\bf C}_\o$ putting
$Q(x(\o),0;z)=Q(x,0;z)$ for any $\o \in \Delta_x$ such that $x(\o)=x$. This
defines an analytic and $\o_1$-periodic---but yet unknown---function. In the
same way, $Q(0,y;z)$ is lifted to the corresponding strip $\Delta_y$ on ${\bf
C}_\o$. The intersection $\Delta_x\cap \Delta_y$ is a non-empty strip (Figure
\ref{Fig_intro}), where both functions are analytic. Since $K(x(\o), y(\o);
z)=0$, it follows from the main equation \eqref{functional_equation} that
\begin{equation*}
    K(x(\o),0;z)Q(x(\o),0;z)+K(0,y(\o);z)Q(0,y(\o);z)-K(0,0;z)Q(0,0;z)-
    x(\o)y(\o)\hspace{-1mm}=\hspace{-1mm}0
\end{equation*}
for any $\o\in \Delta_x \cap \Delta_y$.

This equation allows to continue $Q(x(\o), 0; z)$ to $\Delta_y$ and $Q(0,
y(\o);z)$ to $\Delta_x$, so that both functions become meromorphic and
$\o_1$-periodic on the strip $\Delta=\Delta_x\cup \Delta_y$.
     
           \unitlength=0.5cm
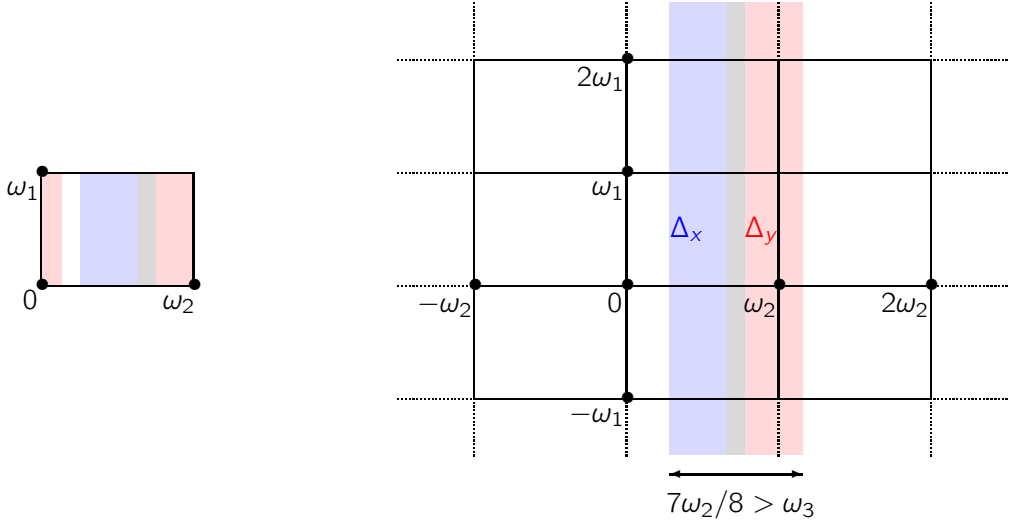
\begin{figure}[t]
\begin{center}\hspace{-18mm}
\begin{tikzpicture}(0,0)
           \linethickness{2mm}
 \draw[draw=none,fill=blue!15] (0.5,-1) -- (0.5,0.5) -- (1.25,0.5) -- (1.25,-1) -- cycle;
   \draw[draw=none,fill=black!15] (1.25,-1) -- (1.25,0.5) -- (1.5,0.5) -- (1.5,-1) -- cycle;
  \draw[draw=none,fill=red!15] (1.5,-1) -- (1.5,0.5) -- (2,0.5) -- (2,-1) -- cycle;
  \draw[draw=none,fill=red!15] (0,-1) -- (0,0.5) -- (0.25,0.5) -- (0.25,-1) -- cycle;
  \draw[draw=none,fill=blue!15] (8.25,-3.25) -- (8.25,2.75) -- (9,2.75) -- (9,-3.25) -- cycle;
  \draw[draw=none,fill=black!15] (9,-3.25) -- (9,2.75) -- (9.25,2.75) -- (9.25,-3.25) -- cycle;
    \draw[draw=none,fill=red!15] (9.25,-3.25) -- (9.25,2.75) -- (10,2.75) -- (10,-3.25) -- cycle;
     \end{tikzpicture}
     
\vspace{-27.5mm}
\hspace{-1.4mm}
\begin{picture}(0,0)
\put(-2,0){\line(1,0){4}}
\put(2,0){\line(1,0){4}}
\put(-6,0){\line(1,0){4}}
\put(-6,-3){\line(1,0){12}}
\put(-6,-3){\line(0,1){6}}
\put(-2,-3){\line(0,1){6}}
\put(2,-3){\line(0,1){6}}
\put(6,-3){\line(0,1){6}}
\put(-2,3){\line(1,0){4}}
\put(2,3){\line(1,0){4}}
\put(-6,3){\line(1,0){4}}
\put(-2,6){\line(1,0){4}}
\put(2,6){\line(1,0){4}}
\put(-6,6){\line(1,0){4}}
\put(-6,3){\line(0,1){3}}
\put(-2,3){\line(0,1){3}}
\put(2,3){\line(0,1){3}}
\put(6,3){\line(0,1){3}}
\put(-6,-3){\dottedline{0.1}(0,0)(0,-1.5)}
\put(-2,-3){\dottedline{0.1}(0,0)(0,-1.5)}
\put(2,-3){\dottedline{0.1}(0,0)(0,-1.5)}
\put(6,-3){\dottedline{0.1}(0,0)(0,-1.5)}
\put(-6,6){\dottedline{0.1}(0,0)(0,1.5)}
\put(-2,6){\dottedline{0.1}(0,0)(0,1.5)}
\put(2,6){\dottedline{0.1}(0,0)(0,1.5)}
\put(6,6){\dottedline{0.1}(0,0)(0,1.5)}
\put(6,6){\dottedline{0.1}(0,0)(2,0)}
\put(-6,-3){\dottedline{0.1}(0,0)(-2,0)}
\put(6,-3){\dottedline{0.1}(0,0)(2,0)}
\put(6,3){\dottedline{0.1}(0,0)(2,0)}
\put(6,0){\dottedline{0.1}(0,0)(2,0)}
\put(-6,6){\dottedline{0.1}(0,0)(-2,0)}
\put(-6,3){\dottedline{0.1}(0,0)(-2,0)}
\put(-6,0){\dottedline{0.1}(0,0)(-2,0)}
\put(-2.16,-0.16){{$\bullet$}}
\put(-6.16,-0.16){{$\bullet$}}
\put(-2.16,2.84){{$\bullet$}}
\put(-2.16,-3.16){{$\bullet$}}
\put(-2.16,5.84){{$\bullet$}}
\put(1.84,-0.16){{$\bullet$}}
\put(5.84,-0.16){{$\bullet$}}
\put(-7.5,-0.7){{$-\o_{2}$}}
\put(-2.5,-0.7){{$0$}}
\put(-2.9,2.4){{$\o_1$}}
\put(-3.5,-3.6){{$-\o_1$}}
\put(-3.3,5.3){{$2\o_1$}}
\put(1.1,-0.7){{$\o_{2}$}}
\put(4.7,-0.7){{$2\o_{2}$}}
\end{picture}
\hspace{-79.5mm}
\begin{picture}(4,-4)
\put(-2,0){\line(0,1){3}}
\put(2,0){\line(0,1){3}}
\put(-2,3){\line(1,0){4}}
\put(-2,0){\line(1,0){4}}
\put(-2.16,-0.16){{$\bullet$}}
\put(-2.16,2.84){{$\bullet$}}
\put(1.84,-0.16){{$\bullet$}}
\put(-2.5,-0.6){{$0$}}
\put(-2.9,2.4){{$\o_1$}}
\put(1.2,-0.6){{$\o_{2}$}}
\put(14.53,-5){\vector(1,0){3.5}}
\put(18.03,-5){\vector(-1,0){3.5}}
\put(14.44,-6){$7\o_2/8>\o_3$}
\put(14.53,1.3){\textcolor{blue}{$\Delta_x$}}
\put(16.49,1.3){\textcolor{red}{$\Delta_y$}}
\end{picture}
\end{center}

 \vspace{28mm}
\caption{Left: the subdomains of the fundamental parallelogram
$[0,\o_1)+[0,\o_2)$ where $\vert x(\o)\vert <1$ (blue), $\vert y(\o)\vert <1$
(red) and $\vert x(\o)\vert,\vert y(\o)\vert <1$ (grey). Right: corresponding
domains $\Delta_x$ (blue) and $\Delta_y$ (red) on the universal covering ${\bf
C}_\o$. The intersection $\Delta_x\cap\Delta_y$ is represented in grey. The
domain $\Delta=\Delta_x\cup \Delta_y$ has length $7\o_2/8$, which is larger
than $\o_3=3\o_2/4$.}
\label{Fig_intro}
\end{figure}

The crucial point of our approach is the following: letting
$r_x(\o)=K(x(\o),0;z)Q(x(\o),0;z)$, we have the key-identity
\begin{equation}
\label{eq:prolongement_introduction}
     r_x(\o-\o_3)=r_x(\o)+f_x(\o),\qquad \forall \o\in {\bf C}_\o,
\end{equation}
where the shift vector $\o_3=3\o_2/4$ (real positive) and the function $f_x$
are explicit (and relatively simple, see \eqref{def_omega_3} and
\eqref{fxfy}). Equation \eqref{eq:prolongement_introduction} has
many useful consequences:
\begin{enumerate}[label={\rm (\Roman{*})},ref={\rm (\Roman{*})}]
\item Due to (a repeated use of)
 \eqref{eq:prolongement_introduction}, the function
$r_x(\o)$ can be {\it meromorphically continued} from its initial
domain of definition $\Delta$ to {\it the
whole plane} 
\begin{equation}
\label{delta}
{\bf C}_\o=\bigcup_{n \in {\bf Z}} \{\Delta+n\o_3\},
\end{equation} 
see
Section \ref{sec:mero}. By projecting back on ${\bf C}_x$, we shall recover all
branches of $Q(x,0;z)$.

\item\label{II}
We shall apply four times \eqref{eq:prolongement_introduction} and
prove the identity
$f_x(\o)+f_x(\o+\o_3)+f_x(\o+2\o_3)+f_x(\o+3\o_3)=0$ (as we remark
in Section \ref{sec:gf_zeta}, it exactly corresponds to the fact
that the orbit sum of Gessel's walks is zero, which was noticed in
\cite[Section 4.2]{BMM}), from where we shall derive that {\it $r_x$
is elliptic with periods $(\o_1,4\o_3)$}, see
Section~\ref{sec:gf_zeta}.

\item Since by \eqref{kn} $4\o_3=3\o_2$ (which is a non-trivial fact,
and means that the group---to be defined in Section
\ref{sec:mero}---of Gessel's model has order $8$), the theory of
transformations of elliptic functions will imply that $r_x$ is algebraic in the
Weierstrass function $\wp$ with periods $\o_1,\o_2$.
This will eventually
yield the {\it algebraicity} of the GF $Q(x,0;z)$. Using a similar
result for $Q(0,y;z)$ and the functional equation
\eqref{functional_equation}, we shall derive in this way the
solution to Problem \ref{enumi:problem_B}, see
Section~\ref{sec:proof_B}.

\item From \eqref{eq:prolongement_introduction} we shall also find the poles
of $r_x$ and the principal parts at them. In general, it is clearly impossible
to deduce the expression of a function from the knowledge of its poles. A
notable exception is constituted by elliptic functions, which is the case of
the function $r_x$ for Gessel walks, see \ref{II} above. From this fact we
shall deduce an explicit expression of $r_x$ in terms of elliptic
$\zeta$-functions. By projection on ${\bf C}_x$, this will give {\it a new
explicit expression of $Q(x,0;z)$ for Gessel walks as an infinite series.} An
analogous result will hold for $Q(0,y;z)$, and \eqref{functional_equation}
will then lead to a new explicit expression for $Q(x,y;z)$, see
Section~\ref{sec:gf_zeta}. This part of the article is inspired
by~\cite{KRnew}. However, it does not rely on results
from~\cite{KRnew}, and it is more elementary.

\item Evaluating the so-obtained expression of $Q(x,0;z)$ at $x=0$ and
performing further simplifications (based on several identities involving
special functions \cite{AS}, and on the theory of the Darboux coverings for
tetrahedral hypergeometric equations \cite{Vidunas08}), we shall obtain the
solution of Problem \ref{enumi:problem_A}, and, in this way, the first human
proof of Gessel's conjecture, see Section \ref{sec:proof_A}. \end{enumerate}

\section{Meromorphic continuation of the generating functions}
\label{sec:mero}

\subsection*{Roadmap}

The aim of Section \ref{sec:mero} is to prove equation
\eqref{eq:prolongement_introduction}, which, as we have seen just above, is
the fundamental starting tool for our analysis. In passing, we shall also
introduce some useful tools for the next sections. Though crucial, this
section does not contain any new result. We thus choose to state the results
and to give some intuition, without proof, and we refer to \cite[Sections
2--5]{KRIHES} and to \cite[Section 2]{KRnew} for full details.

We first properly define the Riemann surface ${\bf T}_z$ in
\eqref{eq:def_T_z}, then we connect it to elliptic functions (in particular,
we obtain the expressions of the functions $x(\o)$ and $y(\o)$ in terms of the
Weierstrass elliptic function~$\wp$). We next introduce the universal covering
of ${\bf T}_z$ (the plane ${\bf C}_\o$) and the group of the walk. We lift the
GFs to the universal covering (this allows us to define the function $r_x(\o)$
in \eqref{eq:prolongement_introduction}). Finally we show how to
meromorphically continue $r_x(\o)$, and we prove the
key-equation~\eqref{eq:prolongement_introduction}.

For brevity, we drop the variable $z$ (which is kept fixed in $(0, 1/4)$) from
the notation when no ambiguity arises, writing for instance $Q(x,y)$ instead
of $Q(x,y;z)$ and ${\bf T}$ instead of~${\bf T}_z$. In
Appendix~\ref{appendix-elliptic}, we gather together a few useful results on
the Weierstrass functions $\wp(z)$ and $\zeta(z)$.

\subsection*{Branch points and Riemann surface ${\bf T}$}
The kernel $K(x,y)$ defined in \eqref{def_kernel} is a quadratic
polynomial with respect to both variables~$x$ and~$y$. The algebraic function $X(y)$
defined by $K(X(y),y)=0$ has thus two branches, and four branch
points that we call $y_i$, $i\in\{1,\ldots ,4\}$\footnote{By definition, a branch point $y_i$ is a point $y\in{\bf C}$ such that the two roots $X(y)$ are equal.}. They are the roots of the discriminant with respect to $x$ of the polynomial $K(x,y)$:
\begin{equation*}
     \widetilde d(y)=(-y)^2-4z^2(y^2+y)(y+1).
\end{equation*}
We have $y_1=0$, $y_4=\infty=1/y_1$, and
\begin{equation*}
     y_2=\frac{1-8z^2-\sqrt{1-16z^2}}{8z^2},\qquad y_3=\frac{1-8z^2+\sqrt{1-16z^2}}{8z^2}=1/y_2,
\end{equation*}
so that $y_1<y_2<y_3<y_4$. Since there are four distinct
branch points, the Riemann surface of $X(y)$, which has the same construction
as the Riemann surface of the algebraic function
\begin{equation}
\label{eq:relation_close_elliptic}
     \sqrt{\widetilde d(y)}=\sqrt{-4z^2(y-y_1)(y-y_2)(y-y_3)},
\end{equation}
is a torus ${\bf T}^y$
(i.e., a Riemann surface of genus $1$). We refer to \cite[Section 4.9]{JS} for the construction of the Riemann surface of the square root of a polynomial, and to~\cite[Section~2]{KRIHES} for this same construction in the context of models of walks in the quarter plane.

The analogous statement holds for the algebraic function $Y(x)$ defined by $K(x,Y(x))=0$. Its
four branch points $x_i$, $i\in\{1,\ldots ,4\}$, are the roots of
\begin{equation}
     \label{dx}
     d(x)=(zx^2-x+z)^2-4z^2x^2.
\end{equation}
They are all real, and numbered so that $x_1<x_2<x_3<x_4$:
\begin{equation*}
     x_1 = \frac{1+2z-\sqrt{1+4z}}{2z},\qquad x_2 = \frac{1-2z-\sqrt{1-4z}}{2z},\qquad x_3=1/x_2,\qquad x_4=1/x_1.
\end{equation*}
The Riemann surface of $Y(x)$ is also a torus ${\bf T}^x$. Since
${\bf T}^x$ and ${\bf T}^y$ are conformally equivalent
(there are two different views of the same surface), in the remainder of our work we shall
consider a single Riemann surface ${\bf T}$ with two different
coverings $x : {\bf T} \to {\bf C}_x$ and $y : {\bf T} \to {\bf C}_y$; see
Figure~\ref{fig:different_levels}.

\unitlength=0.6cm
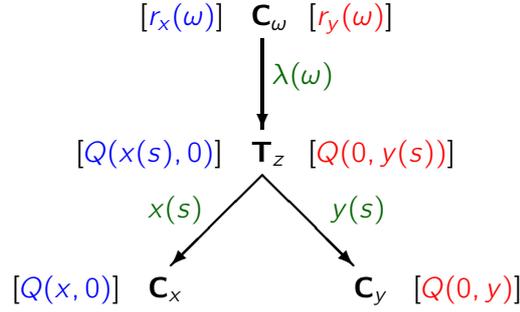
\begin{figure}[t]
  \begin{center}
        \begin{picture}(0,0.5)
    \thicklines
    \put(-0.25,0){${\bf C}_\o$}
    \put(-2.7,0){[\textcolor{blue}{$r_x(\o)$}]}
    \put(1,0){[\textcolor{red}{$r_y(\o)$}]}
    \put(0,-0.3){\vector(0,-1){2}}
    \put(-0.25,-3){${\bf T}_z$}
    \put(-4.1,-3){[\textcolor{blue}{$Q(x(s),0)$}]}
    \put(1,-3){[\textcolor{red}{$Q(0,y(s))$}]}
    \put(0,-3.3){\vector(1,-1){2}}
    \put(0,-3.3){\vector(-1,-1){2}}
    \put(-2.5,-4.25){\textcolor{darkgreen}{$x(s)$}}
    \put(1.5,-4.25){\textcolor{darkgreen}{$y(s)$}}
    \put(0.2,-1.3){\textcolor{darkgreen}{$\lambda(\omega)$}}
    \put(-2.5,-6){${\bf C}_x$}
    \put(2,-6){${\bf C}_y$}
    \put(-5.5,-6){[\textcolor{blue}{$Q(x,0)$}]}
    \put(3.3,-6){[\textcolor{red}{$Q(0,y)$}]}
        \end{picture}
  \end{center}
  \vspace{35mm}
\caption{The GF $Q(x,0)$ is defined on (a subdomain of) the complex plane~${\bf C}_x$. It will be lifted on the Riemann surface ${\bf T}$ as $s\mapsto Q(x(s),0)$, and on the universal covering ${\bf C}_\o$ as $r_x(\o)=K(x(\o),0)Q(x(\o),0)$. The same holds for $Q(0,y)$. We have also represented the projections between the different levels.}
\label{fig:different_levels}
\end{figure}

\subsection*{Connection with elliptic functions}

The torus ${\bf T}$, like any compact Riemann surface of genus $1$,
is isomorphic to a quotient space
${\bf C}/(\omega_1{\bf Z}+\omega_2{\bf Z})$,
where $\omega_1,\omega_2$ are complex numbers linearly
independent on ${\bf R}$, see \cite{JS}.
 This set can obviously be thought as the (fundamental)
  parallelogram $[0, \o_1]+[0, \o_2]$, whose opposed edges
  are identified (here, all parallelograms will be rectangles).
  The periods $\omega_1,\omega_2$ are unique (up to a unimodular transform)
  and are found in \cite[Lemma 3.3.2]{FIM}\footnote{Note a small misprint in Lemma 3.3.2 in \cite{FIM}, namely a (multiplicative) factor of 2 that should be 1; the same holds for \eqref{def_omega_3}.}:
     \begin{equation}
     \label{expression_omega_1_2}
          \omega_1= i\int_{x_1}^{x_2}\frac{\text{d}x}{\sqrt{-d(x)}},\qquad
           \omega_2 = \int_{x_2}^{x_3} \frac{\text{d}x}{\sqrt{d(x)}}.
     \end{equation}
The expression of the periods above cannot be considerably simplified,
but could be written in terms of elliptic integrals, see e.g.~\cite[Eqs. (7.20)--(7.25)]{KRIHES}.

The algebraic curve defined by the kernel $K(x,y)$ can be parametrized using the following uniformization formul\ae,
in terms of the Weierstrass elliptic function $\wp$ with periods
 $\omega_{1},\omega_{2}$ (whose expansion is given in equation
 \eqref{eq:first_time_expansion_wp}):
     \begin{equation}
     \label{eq:uniformization_Gessel}
          \left\{\begin{array}{rcl}
          x(\omega)&=&\displaystyle x_{4}+\frac{d'(x_{4})}
          {\wp(\omega)-d''(x_{4})/6},
          \\ y(\omega)&=&\displaystyle \frac{1}{2a(x(\omega))}
          \left(-b(x(\omega))+\frac{d'(x_{4})
          \wp'(\omega)}{2(\wp(\omega)-
          d''(x_{4})/6)^{2}}\right).\end{array}\right.
     \end{equation}
Here $a(x)=zx^2$ and $b(x)=zx^2-x+z$ are the coefficients of $K(x,y)=a(x)y^2 + b(x)y + c(x)$, and $d(x)$ is defined in \eqref{dx} as $d(x) = b(x)^2-4a(x)c(x)$.
  Due to \eqref{eq:uniformization_Gessel},
  the functions  $x(\o),y(\o)$
      are elliptic functions on the whole~${\bf C}$ with periods $\o_1,\o_2$. By construction
\begin{equation}
\label{kxy0}
      K(x(\o),y(\o))=0, \qquad \forall \o \in {\bf C}.
\end{equation}
     We shall not prove equation
     \eqref{eq:uniformization_Gessel}
     here (we refer to \cite[Lemma 3.3.1]{FIM} for details).
     Let us simply point out that it corresponds
     to the rewriting of $K(x,y)=0$ as $(2a(x)y+b(x))^2=d(x)$, then as $w^2=z^2(x-x_1)(x-x_2)(x-x_3)(x-x_4)$, 
      and finally as the classical identity involving elliptic functions $(\wp')^2=4\wp^3-g_2\wp-g_3=4(\wp-e_1)(\wp-e_2)(\wp-e_3)$.

\subsection*{Universal covering}
The universal covering of ${\bf T}$ has the form $({\bf C},
\lambda)$, where ${\bf C}$ is the complex plane that can be viewed
as the union of infinitely many parallelograms
     \begin{equation*}
          \Pi_{m,n}=\omega_1[m,m+1)+\omega_2[n,n+1),\qquad m,n \in {\bf Z},
     \end{equation*}
which are glued together and $\lambda : {\bf C}\to {\bf T}$ is a
non-branching covering map (Figure \ref{fig:different_levels}).
 This is a standard fact on Riemann surfaces, see,
  e.g., \cite[Section 4.19]{JS}.
     For any $\o \in {\bf C}$ such that
  $\lambda \o=s \in {\bf T}$,
  we have $x(\o)=x(s)$ and $y(\o)=y(s)$.
 The expression of $\lambda \o$ is very simple:
  it equals the unique $s$ in the rectangle $[0,\o_1)+[0,\o_2)$
   such that $\o=s+m\o_1+n\o_2$ with some $m,n\in {\bf Z}$.


Furthermore, since each parallelogram $\Pi_{m,n}$ represents
a torus ${\bf T}$ composed of two complex spheres,
the function $x(\o)$ (resp.\ $y(\o)$)
takes each value of ${\bf C}\cup \{\infty\}$ twice within this parallelogram,
 except for the branch points $x_i$, $i\in\{1,\ldots ,4\}$
 (resp.\ $y_i$, $i\in\{1,\ldots ,4\}$).
The points $\o_{x_i} \in \Pi_{0,0}$ such that $x(\o_{x_i})=x_i$,
$i\in\{1,\ldots,4\}$, are represented in
Figure~\ref{The_fundamental_parallelogram}. They are equal to
\begin{equation*}
      \o_{x_1}=\o_2/2,\qquad \o_{x_2}=(\o_1+\o_2)/2,\qquad \o_{x_3}=\o_1/2,\qquad \o_{x_4}=0.
\end{equation*}
   The points $\o_{y_i}$ such that $y(\o_{y_i})=y_i$ are just the
shifts of $\o_{x_i}$ by a real vector $\o_3/2$ (to be defined
below, in equation~\eqref{def_omega_3}): $\o_{y_i}=\o_{x_i}+\o_3/2$ for $i\in\{1,\ldots,4\}$, see
also Figure \ref{The_fundamental_parallelogram}. We refer to
\cite[Chapter 3]{FIM} and to \cite{KRG,Ra} for proofs of these
facts.
   The vector $\o_3$ is defined as in~\cite[Lemma 3.3.3]{FIM}:
   \begin{equation}
     \label{def_omega_3}
          \omega_{3}= \int_{-\infty}^{x_{1}}
          \frac{\text{d}x}{\sqrt{d(x)}}.
     \end{equation}
For Gessel's model we have the following relation
\cite[Proposition 14]{KRG}, which holds for all $z\in(0,1/4)$:
     \begin{equation}
     \label{kn}
         {\omega_3}/{\omega_2} = {3}/{4}.
     \end{equation}
The identity above contains a lot of informations:
 it turns out that for any model of walks in the quarter plane,
 the quantity ${\omega_3}/{\omega_2}$
 is a rational number (independent of $z$) if and only if a certain group
 (to be defined in the next section) is finite,
 see \cite[Eq.~(4.1.11)]{FIM}.
 Equation \eqref{kn} thus readily implies that
 Gessel's group is finite (of order $8$).
  Although we shall not use this result here,
  let us also mention that if ${\omega_3}/{\omega_2}$ is rational,
  then the solution of the functional equation \eqref{functional_equation}
  (i.e., the generating function \eqref{eq:def_generating_function} of interest)
   is D-finite (with respect to~$x$ and~$y$), see \cite[Theorem 2.1]{FR}.
    On the other hand, the relation \eqref{kn}
    does not imply, a priori,
    that the generating function \eqref{eq:def_generating_function}
    is algebraic.

      \unitlength=0.52cm
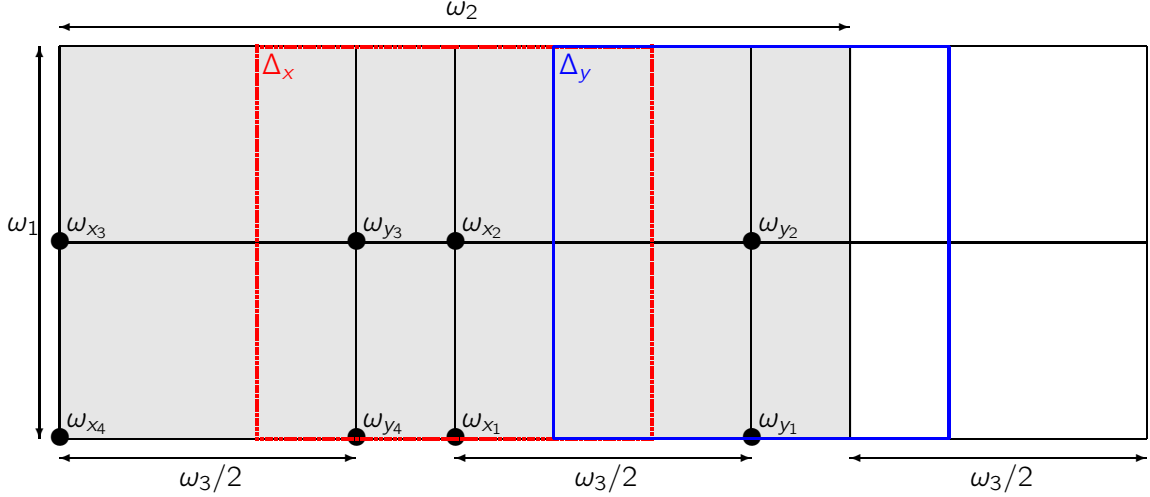
\begin{figure}[t]
    \vspace{5mm}
  \begin{center}
\begin{tabular}{cccc}
    \hspace{-7.3cm}

 \begin{tikzpicture}(19.55,6.5)
           \linethickness{2mm}
 \draw[draw=none,fill=black!10] (-5,5.2) -- (5.42,5.2) -- (5.42,0) -- (-5,0) -- cycle;
     \end{tikzpicture}

\hspace{-53.5mm}\begin{picture}(0,0)(0,0)
\put(-10,0){\line(1,0){20}}
\put(-10,10){\line(1,0){20}}
\put(-10,5){\line(1,0){20}}
\put(-10,0){\line(0,1){10}}
\put(10,0){\line(0,1){10}}
\put(0,0){\line(0,1){10}}
\put(7.5,0){\line(0,1){10}}
\put(-2.5,0){\line(0,1){10}}
\put(-10.27,-0.27){{\LARGE$\bullet$}}
\put(-2.77,-0.27){{\LARGE$\bullet$}}
\put(-0.27,-0.27){{\LARGE$\bullet$}}
\put(7.23,-0.27){{\LARGE$\bullet$}}
\put(-10.27,4.73){{\LARGE$\bullet$}}
\put(-2.77,4.73){{\LARGE$\bullet$}}
\put(-0.27,4.73){{\LARGE$\bullet$}}
\put(7.23,4.73){{\LARGE$\bullet$}}
\put(-9.8,0.3){{$\o_{x_4}$}}
\put(-2.3,0.3){{$\o_{y_4}$}}
\put(0.2,0.3){{$\o_{x_1}$}}
\put(7.7,0.3){{$\o_{y_1}$}}
\put(-9.8,5.3){{$\o_{x_3}$}}
\put(-2.3,5.3){{$\o_{y_3}$}}
\put(0.2,5.3){{$\o_{x_2}$}}
\put(7.7,5.3){{$\o_{y_2}$}}
\put(-10,-0.5){\vector(1,0){7.5}}
\put(-2.5,-0.5){\vector(-1,0){7.5}}
\put(-6.95,-1.2){{$\o_{3}/2$}}
\put(0,-0.5){\vector(1,0){7.5}}
\put(7.5,-0.5){\vector(-1,0){7.5}}
\put(3.05,-1.2){{$\o_{3}/2$}}
\put(-10,10.5){\vector(1,0){20}}
\put(10,10.5){\vector(-1,0){20}}
\put(-0.2,10.8){{$\o_{2}$}}
\put(-10.5,0){\vector(0,1){10}}
\put(-10.5,10){\vector(0,-1){10}}
\put(-11.3,5.3){{$\o_{1}$}}
\put(10,0){\line(1,0){7.5}}
\put(10,5){\line(1,0){7.5}}
\put(10,10){\line(1,0){7.5}}
\put(17.5,10){\line(0,-1){10}}
\put(13.05,-1.2){{$\o_{3}/2$}}
\put(10,-0.5){\vector(1,0){7.5}}
\put(17.5,-0.5){\vector(-1,0){7.5}}

 \linethickness{0.4mm}
\put(-5,0){\textcolor{red}{\dottedline{0.12}(0,0)(0,10)}}
\put(5,0){\textcolor{red}{\dottedline{0.12}(0,0)(0,10)}}
\put(-5,10){\textcolor{red}{\dottedline{0.12}(0,0)(10,0)}}
\put(-5,0){\textcolor{red}{\dottedline{0.12}(0,0)(10,0)}}

 \linethickness{0.3mm}

\put(2.5,0){\textcolor{blue}{\line(0,1){10}}}
\put(12.5,0){\textcolor{blue}{\line(0,1){10}}}
\put(2.5,10){\textcolor{blue}{\line(1,0){10}}}
\put(2.5,0){\textcolor{blue}{\line(1,0){10}}}

\put(-4.88,9.3){\textcolor{red}{{$\Delta_x$}}}
\put(2.62,9.3){\textcolor{blue}{{$\Delta_y$}}}
\end{picture}
    \end{tabular}
  \end{center}

  \vspace{10mm}
\caption{The fundamental parallelograms for
the functions $r_x(\o)$ and $r_y(\o)$, namely, $\Pi_{0,0}=\o_1[0,1)+\o_2[0,1)$ (in grey) and $\Pi_{0,0}+\o_3/2=\o_1[0,1)+\o_2[0,1)+\o_3/2$, and some important points and domains on them.}
\label{The_fundamental_parallelogram}
\end{figure}

\subsection*{Galois automorphisms and group of the walk}

It is easy to see that the 
birational transformations 
$\xi$ and $\eta$ of ${\bf C}^2$ defined by
     \begin{equation}
     \label{xieta}
          \xi(x,y)= \left(x,\frac{1}{x^2 y}\right),\qquad \eta(x,y)=\left(\frac{1}{xy},y\right)
 \end{equation}
leave invariant the quantity $\sum_{(i,j)\in \mathfrak{G} } x^i y^j$
(and therefore also the set ${\bf T}$ in \eqref{eq:def_T_z}
for any fixed $z\in (0,1/4)$).
 They span a group $\langle \xi,\eta\rangle$
 of birational transformations of ${\bf C}^2$, which is a dihedral group,
  since
\begin{equation}
\label{id}
     \xi^2=\eta^2=\text{id}.
\end{equation}
It is of order $8$, see \cite{BMM}. 

This group was first defined in a probabilistic context by Malyshev
\cite{MA,MAL,MALY}; it was introduced for the combinatorics of walks
with small steps in the quarter plane by Bousquet-M\'elou
\cite{BM02,BM05},  and more systematically by Bousquet-M\'elou and
Mishna \cite{BMM}. In \eqref{xieta}, it is defined as a group on
${\bf C}^2={\bf C}_x\times {\bf C}_y$, i.e., at the bottom level of
Figure \ref{fig:different_levels}. We now lift it to the upper
levels of Figure \ref{fig:different_levels}, that is, to ${\bf T}_z$
and ${\bf C}_\o$. Our final objective is to demonstrate the
following result:
\begin{equation}
\label{hatxieta}
           \xi \o=-\o+\o_1+\o_2,\qquad  \eta \o=-\o+\o_1+\o_2+\o_3,\qquad \forall \o\in{\bf C}.
\end{equation}
This equation illustrates the fact that the universal covering
is a natural object: while the expressions of the elements of the
group were rather complicated in \eqref{xieta}, they are now just
affine functions.

\begin{proof}[Proof of Equation \eqref{hatxieta}]
First, we lift the elements of the group to the intermediate level ${\bf T}$ as the
restriction of $\langle \xi, \eta \rangle $ on ${\bf T}$. Namely,
any point $s \in {\bf T}$ admits the two ``coordinates''
$(x(s),y(s))$, which satisfy
$K(x(s), y(s))=0$  by construction. For
any $s \in {\bf T}$, there exists a unique $s'$ (resp.\ $s''$) such
that $x(s')=x(s)$ (resp.\ $y(s'')=y(s)$). The values $x(s),x(s')$
(resp.\ $y(s),y(s'')$) are the two roots of the second degree equation
 $K(x, y(s))=0$ (resp.\ $K(x(s), y)=0$) in $x$ (resp.\
$y$). The automorphism $ \xi : {\bf T} \to {\bf T}$ (resp.\ $ \eta :
{\bf T} \to{\bf T}$) is defined by the identity $\xi s =s'$
(resp.\ $\eta s=s''$) and is called a Galois automorphism,
following the terminology of \cite{MA,MAL,MALY,FIM}. Clearly by
\eqref{xieta} and \eqref{id}, we have, for any $s\in {\bf T}$,
     \begin{equation*}
          x(\xi s)=x(s),\ \  y(\xi s)=\frac{1}{x^2(s)y(s)},\ \
          x( \eta s)=\frac{1}{y(s)x(s)}, \ \  y(\eta
          s)=y(s),\ \  \xi^2(s)=\eta^2(s)=s.
     \end{equation*}
Finally $ \xi s=s$ (resp.\ $\eta s=s$) if and only if $x(s)=x_i$,
$i\in\{1,\ldots ,4\}$ (resp.\ $y(s)=y_i$, for some $i\in\{1,\ldots
,4\}$).

  There are many ways to lift $\xi$ and $\eta$ from ${\bf T}$ to the universal
  covering ${\bf C}$. For any of them
  $\xi(\o_{x_i})=\o_{x_i}+n\o_1+m\o_2$,
$\eta(\o_{y_i})=\o_{y_i}+k\o_1+l\o_2$ for some $n,m,k,l\in {\bf Z}$.
   There should also exist  constants $p,q,r,s \in {\bf Z}$ such that
 $\xi^2={\rm id}+p\o_1+q\o_2$ and $\eta^2={\rm id}+r\o_1+s\o_2$.

       We follow the way of \cite{FIM} and
  \cite{KRIHES}, lifting them on  ${\bf C}$ in such a way that
$\o_{x_2}$ and $\o_{y_2}$ are their fixed points, respectively (see
Figure \ref{The_fundamental_parallelogram}).
  It follows immediately that $p,q,r,s=0$.
    Since any automorphism of ${\bf C}_\o$ has the form
       $\alpha \o +\beta$ with $\alpha, \beta \in {\bf C}$, 
  the relations $\xi^2={\rm id}$ and $\xi(\o_{x_2})=\o_{x_2}$
     (resp.\ $\eta^2={\rm id}$ and $\eta(\o_{y_2})=\o_{y_2}$),
   lead to $\alpha=-1$ and $\beta=2\o_{x_2}$ (resp.\ $\alpha=-1$ and $\beta=2\o_{y_2}$).
 We obtain equation \eqref{hatxieta}, recalling that
$\o_1+\o_2=2\o_{x_2}$ and that $\o_1+\o_2+\o_3=2\o_{y_2}$. The proof of 
equation~\eqref{hatxieta} is thus completed.
\end{proof}

By construction the elements of the group satisfy
     \begin{equation*}
          x(\xi \o)=x(\o),\qquad y(\eta\o)=y(\o), \qquad \forall \o\in{\bf C}.
     \end{equation*}
Finally, by \eqref{hatxieta}
\begin{equation}
\label{hatxi} \eta\xi \o=\o+\o_3,\qquad \xi\eta\o=\o-\o_3, \qquad \forall \o\in{\bf C}.
\end{equation}

\subsection*{Lifting of the GFs on the universal covering}

   The functions $Q(x,0)$ and $Q(0,y)$ can be
   lifted on their respective natural domains of definition on ${\bf
  T}$ and next on the corresponding domains of the universal
  covering ${\bf C}$, namely $\{\o\in{\bf C} : |x(\o)|<1\}$
  and  $\{\o \in{\bf C}: |y(\o)|<1\}$.
   This lifting procedure is illustrated in Figure \ref{fig:different_levels}.
 The first level (at the bottom) represents the complex planes
 ${\bf C}_x$ and ${\bf C}_y$, where $Q(x,0)$ and $Q(0,y)$ are defined in
  $\{|x|<1\}$ and $\{|y|<1\}$.
 The second level, where the variables $x$ and $y$ are not independent anymore,
   is given by ${\bf T}$. The third level is ${\bf C}$, the universal
   covering of  ${\bf T}$.
     All this construction has been first elaborated by  Malyshev \cite{MA}
   for stationary probability GFs of random walks
   in ${\bf N}^2$, and has been further developed in \cite{FIM}
   and in \cite{KRIHES} in a combinatorial context.

The domains
\begin{equation*}
     \{\o\in{\bf C} : |x(\o)|<1\},\qquad\{\o\in{\bf C}: |y(\o)|<1\}
\end{equation*}
consist of infinitely many curvilinear strips, which differ from
translations by multiples of $\o_2$. We denote by $\Delta_x$
(resp.\ $\Delta_y$) the strip that is within $\cup_{m \in {\bf
Z}}\Pi_{m,0}$ (resp.\ $\cup_{m \in {\bf Z}} \Pi_{m,0}+\o_3/2$). The
domain $\Delta_x$ (resp.\ $\Delta_y$) is delimited by vertical
lines, see \cite[Proposition 26]{KRG}, and is represented in Figure~\ref{The_fundamental_parallelogram}. We notice that the function
$Q(x(\o),0)$ (resp.\ $Q(0,y(\o))$) is well defined in $\Delta_x$
(resp.\ $\Delta_y$), by its expression
\eqref{eq:def_generating_function} as a GF. Let us define
\begin{equation}
\label{rxryrxry}\left\{
\begin{array}{rcl}
     r_x(\o)\ =&\! K(x(\o),0)Q(x(\o), 0),\qquad \forall \o \in \Delta_x,\\
     r_y(\o)\ =&\! K(0, y(\o))Q(0, y(\o)),\qquad \forall \o \in
     \Delta_y.
     \end{array}\right.
\end{equation}
The domain $\Delta_x \cap \Delta_y$ is a non-empty open strip, see
Figure \ref{The_fundamental_parallelogram}.
 It follows from \eqref{functional_equation} and \eqref{kxy0} that
\begin{equation}
\label{kze}
     r_x(\o)+r_y(\o)-K(0,0)Q(0,0)-x(\o)y(\o)=0,\qquad \forall \o \in \Delta_x \cap \Delta_y.
\end{equation}

\subsection*{Meromorphic continuation of the GFs on the universal covering}
Let $\Delta =\Delta_x \cup \Delta_y$. Due to \eqref{kze}, the
functions $r_x(\o)$ and $r_y(\o)$ can be continued as meromorphic
functions on the whole domain~$\Delta$, by setting
\begin{align}
&r_x(\o)= -r_y(\o)+K(0,0)Q(0,0)+x(\o)y(\o), \qquad
\forall \o \in
\Delta_y,\nonumber\\
 &r_y(\o)= -r_x(\o)+K(0,0)Q(0,0)+x(\o)y(\o),\qquad \forall \o
\in \Delta_x.\label{rDrD}
\end{align}
To continue the functions from $\Delta$ to the whole complex plane ${\bf C}$,
 we first notice that
         $ \cup_{n\in {\bf Z}}(\Delta+n\omega_3) ={\bf C}$ (see \eqref{delta}),
as proved in \cite{FIM, KRIHES} and illustrated in Figure
\ref{The_fundamental_parallelogram}.  

Let us define
\begin{equation}
\label{fxfy}\left\{
\begin{array}{rcl}
     f_x(\o)&=&y(\omega)[x(-\omega+\o_1+\o_2+\o_3)-x(\omega)],\\
     f_y(\o)&= &x(\omega)[y(-\omega+\o_1+\o_2)-y(\omega)].
     \end{array}\right.
\end{equation}

\begin{lem}[\cite{KRIHES}]
\label{thm2}
The functions $r_x(\omega)$ and $r_y(\omega)$ can be continued meromorphically to the whole of ${\bf C}$. Further, for any $\omega \in {\bf C}$, we have
  \begin{align}
  &r_x(\omega)+r_y(\omega)-K(0,0)Q(0,0)-x(\omega)y(\omega) =0,
       \label{sqs2}
       \\
       &r_x(\omega-\omega_3)=r_x(\omega)+f_x(\o),
       \nonumber
       \\
       &r_y(\omega+\omega_3)=r_y(\omega)+f_y(\o),
       \label{cont1}
       \\
       &\left\{\begin{array}{cc}
       r_x( \xi \omega)=r_x(\omega), \\r_y( \eta\omega)=r_y(\omega),
       \end{array}\right.
       \label{xietabad}
       \\
       &\left\{\begin{array}{cc}
       r_x(\omega+\omega_1)=r_x(\omega), \\
       r_y(\omega+\omega_1)=r_y(\omega).
       \end{array}\right.
       \label{buzz}
  \end{align}
   \end{lem}
\begin{proof}[Sketch of the proof]
We shall not prove Lemma \ref{thm2} above in full details (for that
we refer to the proof of \cite[Theorem 4]{KRIHES}). However, we give
some intuitions on the above identities. First, \eqref{sqs2} is
simply the translation of the functional equation
\eqref{functional_equation}. Second, \eqref{xietabad} comes from the
fact that $x(\o)$ (resp.\ $y(\o)$) is invariant by $\xi$ (resp.\
$\eta$). Equation \eqref{buzz} follows from the $\o_1$-periodicity
of the functions $x(\o)$ and $y(\o)$. Let us give some more details
for \eqref{cont1}. Evaluate \eqref{sqs2} at $\o$ and $\xi\o$, and
make the subtraction of the two identities so obtained. We deduce
that $r_y(\xi\o)-r_y(\o)=x(\o)[y(\xi\o)-y(\o)]$. We conclude by
\eqref{xietabad}, since $r_y(\xi\o)=r_y(\eta\xi\o)=r_y(\o+\o_3)$,
where the last equality follows from \eqref{hatxi}.
\end{proof}

\section{Generating functions in terms of Weierstrass zeta functions}
\label{sec:gf_zeta}

\subsection*{Statements of results} The aim of this section is to prove that
the generating function for Gessel walks can be expressed in terms of the
Weierstrass zeta function (Theorems~\ref{thm:first_Gessel_simplified}
and~\ref{thm:Gessel-bi} below). To formulate them, we need to recall some
notation: let $\o_1,\o_2$ be the periods defined in
\eqref{expression_omega_1_2}, and let $\zeta_{1,3}$ be the Weierstrass zeta
function with periods $\o_1,3\o_2$, see Appendix~\ref{appendix-elliptic} for
its definition and some of its properties. We shall prove the following
results:

\begin{thm}
\label{thm:first_Gessel_simplified}
{For any $z\in(0,1/4)$} we have
\begin{align}
\label{eq:first_Gessel_simplified}
     &Q(0,0;z)=\\&\frac{\zeta_{1,3}(\o_2/4)-3\zeta_{1,3}(2\o_2/4)+2\zeta_{1,3}(3\o_2/4)+3\zeta_{1,3}(4\o_2/4)-5\zeta_{1,3}(5\o_2/4)+2\zeta_{1,3}(6\o_2/4)}{2z^2}.\nonumber
\end{align}
\end{thm}

\begin{thm}
\label{thm:Gessel-bi}
We have, for all $\o\in{\bf C}$,
\begin{align}
     r_y(\o) = c\, +&\frac{1}{2z}\zeta_{1,3}(\o-(1/8)\o_2)-\frac{1}{2z}\zeta_{1,3}(\o-(3/8)\o_2)\nonumber\\
                      +&\frac{1}{2z}\zeta_{1,3}(\o-(1+3/8)\o_2)-\frac{1}{2z}\zeta_{1,3}(\o-(1+5/8)\o_2)\nonumber\\
                      -&\frac{1}{2z}\zeta_{1,3}(\o-(1+7/8)\o_2)+\frac{1}{z}\zeta_{1,3}(\o-(2+1/8)\o_2)\nonumber\\
                      -&\frac{1}{z}\zeta_{1,3}(\o-(2+5/8)\o_2)+\frac{1}{2z}\zeta_{1,3}(\o-(2+7/8)\o_2),\label{zeta}
\end{align}
where $c$ is a constant (depending only on $z$).
\end{thm}

Note that the constant $c$ in the statement of Theorem~\ref{thm:Gessel-bi} can
be made explicit. In fact, the point $\o_0^y=7\o_2/8 \in \Delta_y$ is such
that $y(\o_0^y)=0$ (see Lemma \ref{lem:lemKRG} below). Hence the value of
$r_y(7\o_2/8)=K(0,y(7\o_2/8))Q(0,y(7\o_2/8 ))$ is equal to
$K(0,0)Q(0,0)=zQ(0,0)$ which is given by
Theorem~\ref{thm:first_Gessel_simplified}. Thus $c$ is equal to
$zQ(0,0)-\widehat \zeta_{1,3}(7\o_2/8)$, where $\widehat \zeta_{1,3}(\o)$ is
the sum of the eight $\zeta$-functions in \eqref{zeta}.

An expression similar to \eqref{zeta} holds for $r_x(\o)$ (for a different
constant $c$). There are two ways to obtain this expression: the first one
consists in doing the same analysis as for $r_y$; the second one is to express
$r_x$ from equation \eqref{sqs2} in terms of $r_y$ and to
apply~Theorems~\ref{thm:first_Gessel_simplified} and~\ref{thm:Gessel-bi}. In
terms of $\zeta$-functions, the results of both approaches are rigorously the
same.

Theorems \ref{thm:first_Gessel_simplified} and \ref{thm:Gessel-bi} are crucial
for the remainder of the article. We shall explain in
Section~\ref{sec:proof_A} how to deduce from Theorem
\ref{thm:first_Gessel_simplified} a proof of Gessel's conjecture (Problem
\ref{enumi:problem_A}). Then, in Section \ref{sec:proof_B}, we deduce from
Theorem~\ref{thm:Gessel-bi} the algebraicity of $Q(0,y;z)$ and $Q(x,0;z)$.
Using the functional equation~\eqref{functional_equation}, we shall then
obtain the algebraicity of the complete generating function $Q(x,y;z)$
(Problem \ref{enumi:problem_B}).

\subsection*{Preliminary results}
The poles of the function $f_y$ defined in equation~\eqref{fxfy} will play a crucial role in our analysis.
They are given in the lemma hereafter.

\begin{lem}
\label{lemma:f_y_gessel}
In the fundamental parallelogram $\o_1[0,1)+\o_2[0,1)$, the function $f_y$ has poles at $\o_2/8$, $3\o_2/8$, $5\o_2/8$ and $7\o_2/8$. These poles are simple, with residues equal to $-1/(2z)$, $1/(2z)$, $1/(2z)$ and $-1/(2z)$, respectively.
\end{lem}
Before proving Lemma \ref{lemma:f_y_gessel}, we recall from \cite[Lemma 28]{KRG} the following result, dealing with the zeros and poles of $x(\o)$ and $y(\o)$:
\begin{lem}[\cite{KRG}]
\label{lem:lemKRG} In the fundamental parallelogram
$\o_1[0,1)+\o_2[0,1)$, the only poles of $x$ (of order one) are at
$\o_2/8$, $7\o_2/8$, and its only zeros (of order one) are at
$3\o_2/8$, $5\o_2/8$. The only pole of $y$ (of order two) is at
$3\o_2/8$, and its only zero (of order two) is at $7\o_2/8$.
\end{lem}

\begin{proof}[Sketch of the proof of Lemma \ref{lem:lemKRG}]
Expressions for $x(\o)$ and $y(\o)$ are written down
in \eqref{eq:uniformization_Gessel}.
To show that $x$ has a pole of order $1$ at $\o_2/8$
 it is enough to prove that $\wp(\o_2/8)=d''(x_4)/6$,
  see again \eqref{eq:uniformization_Gessel}.
  Such computations follow from the fact that both quantities are
  known in terms of the variable $z$.
  This is clear for the right-hand side.
  For $\wp(\o_2/8)$ we can use $\wp(\o_2/4)=(1+4z^2)/3$
  (see the proof of Lemma \ref{lem:three_formulae})
   and then the bisection formula \ref{bisection}.
   We do not pursue the computations in more details.
   We would proceed similarly for the other poles and zeros.
\end{proof}

\begin{proof}[Proof of Lemma \ref{lemma:f_y_gessel}]
Using the definition \eqref{fxfy} of the function $f_y(\o)$ and
formulas~\eqref{eq:uniformization_Gessel}, we derive that
\begin{equation*}
     f_y(\o) = \frac{1}{2z}\frac{x'(\o)}{x(\o)}.
\end{equation*}
Indeed, we have
\begin{equation*}
     {f_y(\o)=x(\o)[y(-\o)-y(\o)]}=\frac{x(\o)}{2a(x(\o))}\left(-\frac{d'(x_{4})
          \wp'(\omega)}{(\wp(\omega)-
          d''(x_{4})/6)^{2}}\right)=\frac{1}{2z}\frac{x'(\o)}{x(\o)}.
\end{equation*}
Above, we have used the identities $y(-\o+\o_1+\o_2)=y(-\o)$, $x(-\o)=-x(\o)$, $\wp'(-\o)=-\wp'(\o)$ and $a(x)=z x^2$.

 Accordingly, if $x(\o)$ has a
simple zero (resp.\ a simple pole) at $\o_0$, then $f_y(\o)$ has a
simple pole at $\o_0$, with residue $1/(2z)$ (resp.\ $-1/(2z)$).
Lemma \ref{lemma:f_y_gessel} then follows from Lemma
\ref{lem:lemKRG}.
\end{proof}

The following lemma will shorten the proof of Theorem \ref{thm:Gessel-bi}.
\begin{lem}
\label{lem:elliptic}
The function $r_y$ is elliptic with periods $\o_1,3\o_2$.
\end{lem}
\begin{proof}
The function $r_y$ is meromorphic and $\o_1$-periodic due to \eqref{buzz}.
  Further, by Lemma \ref{thm2},
\begin{equation}
\label{ooo}
r_y(\o+4\o_3)-r_y(\o)=f_y(\o)+f_y(\o+\o_3)+f_y(\o+2\o_3)+f_y(\o+3\o_3), \qquad \forall \o \in {\bf C}.
\end{equation}
 We start by showing that
   the elliptic function
   \begin{equation*}
   \mathcal{O}(\o)=\sum_{k=0}^{3}f_y(\o+k\o_3)
   \end{equation*} has no poles on ${\bf
   C}$. As $\mathcal{O}(\o)$ is $(\o_1,\o_2)$-periodic,
     it suffices to verify this on the parallelogram $[0,\o_1)+[0,\o_2)$.
    Since $\mathcal{O}(\o)$ is also $\o_3$-periodic (this follows immediately
  from $4\o_3=3\o_2$ and the $\o_2$-periodicity of $f_y(\o)$), it is enough
   to check that the poles of $f_y(\o)$ are not those of~$\mathcal{O}(\o)$.
    The function $f_y(\o)$ has four poles in the main parallelogram, at
    $\o_2/8$, $3\o_2/8$, $5\o_2/8$ and $7\o_2/8$
    (Lemma~\ref{lemma:f_y_gessel}).
    Since $\mathcal{O}(\o)$ is also $\o_2-\o_3=\o_2/4$-periodic, 
    it remains to check that $\o_2/8$ is a removable singularity.
     This  is an elementary
       verification using the residues
     of $f_y(\o)$ at its poles, which are given in
     Lemma~\ref{lemma:f_y_gessel}.

      Hence, with property \ref{elliptic_poles_0}
    of Lemma \ref{lemma_properties_wp}, $\mathcal{O}(\o)$ must be
    a constant $c$, so that with \eqref{ooo}
    $r_y(\o+4\o_3)=r_y(\o)+c$ for all $\o \in {\bf C}$.
   In particular, evaluating the previous equality at $\o_{y_2}$ and
    $\o_{y_2}-4\o_3$ and summing the two identities so obtained gives
    $r_y(\o_{y_2}-4\o_3)+2c=r_y(\o_{y_2}+4\o_3)$.
   But in view of \eqref{xietabad},
   $r_y(\o_{y_2}-4\o_3)=r_y(\o_{y_2}+4\o_3)$,
   since $\eta\omega=-\omega+2\omega_{y_2}$, and then $c=0$.

        It follows that $r_y(\o)$ is also $4\o_3=3\o_2$-periodic,
     and thus elliptic with periods $\o_1,3\o_2$.
   \end{proof}

\begin{rem} 
\label{rem:orbit} 
Note that the fact that $\mathcal{O}(\o)$ is identically zero also follows
from  \cite[Proposition 10]{KRnew}, which gives an easy necessary and
sufficient condition for $\mathcal{O}(\o)$ to be zero (equivalently, for $r_x$ and $r_y$ to be elliptic, or for the orbit sum of Bousquet-M\'elou and Mishna to be zero): the poles of
$x(\o)$ and $y(\o)$ in the fundamental parallelogram should not be
poles of $\mathcal{O}(\o)$. In the case of Gessel's walks, by Lemma
\ref{lem:lemKRG},
   it is reduced to checking that the points $\o_2/8$,
   $3\o_2/8$ and $7\o_2/8$ are not poles of $\mathcal{O}(\o)$.
This is immediate  by Lemma~\ref{lemma:f_y_gessel}.

 Finally,
Lemma~\ref{lem:elliptic} is
 proved in \cite[Proposition 11]{KRIHES} as well, using the
representation of $\mathcal{O}(\o)$ as the so-called orbit-sum:
\begin{align*}
          \mathcal{O}(\o)=&\sum_{1\leq k\leq 4} (xy) (\omega +k \omega_3)-(xy) ( \eta (\o+k\o_3))\\
          =&\sum_{1\leq k\leq 4} (xy)(( \eta  \xi)^k \o)-(xy)(\xi ( \eta  \xi)^{k-1}\o)\\
          =&\sum_{\theta \in\langle \xi,\eta\rangle}(-1)^{\theta} xy(\theta(\omega)),
     \end{align*}
  where  $(-1)^{\theta}$ is the signature of $\theta$, i.e., $(-1)^{\theta}=(-1)^{\ell(\theta)}$, where $\ell(\theta)$ is the smallest $\ell$ for which we can
write $\theta = \theta_1\circ \cdots \circ \theta_\ell$, with
$\theta_1,\ldots,\theta_\ell$ equal to $\xi$ or $\eta$.
\end{rem}

\subsection*{Proof of Theorems \ref{thm:first_Gessel_simplified} and \ref{thm:Gessel-bi}}

In order to prove Theorem \ref{thm:Gessel-bi}, we could use \cite[Theorem 6]{KRnew}, which gives the poles and the principal parts at these poles of $r_y$ in terms of the function $f_y$, for any model of walks with small steps in the quarter plane (and rational $\o_2/\o_3$---which is the case here, see \eqref{kn}). However, we prefer adopting a simpler and more direct approach, which is based on our key-equation \eqref{eq:prolongement_introduction}. We shall then deduce Theorem \ref{thm:first_Gessel_simplified} from Theorem \ref{thm:Gessel-bi}.

\begin{proof}[Proof of Theorem \ref{thm:Gessel-bi}]
Since $r_y$ is elliptic with periods $\o_1,3\o_2$ (Lemma \ref{lem:elliptic}),
 and since any elliptic function is characterized by its poles in a fundamental parallelogram,
 it suffices to find the poles of $r_y$ in
 $\o_1[-1/2,1/2)+\o_2[1/8,25/8)$. We shall consider the decomposition
\begin{align}
\nonumber
      \frac{\o_1}{2}[-1,1)+\frac{\o_2}{8}[1,25)=
      &\Bigl\{\frac{\o_1}{2}[-1,1)+\frac{\o_2}{8}[5,9)\Bigl\}
      \cup \Bigl\{\frac{\o_1}{2}[-1,1)+\frac{\o_2}{8}[2,5)\Bigl\}
      \\\nonumber\cup&\Bigl\{\frac{\o_1}{2}[-1,1)+\frac{\o_2}{8}[9,15)\Bigl\}\cup\Bigl\{\frac{\o_1}{2}[-1,1)+\frac{\o_2}{8}[15,21)\Bigl\}\\\cup&\Bigl\{\frac{\o_1}{2}[-1,1)+\frac{\o_2}{8}[21,25)\Bigl\}\cup\Bigl\{\frac{\o_1}{2}[-1,1)+\frac{\o_2}{8}[1,2)\Bigl\},\label{decompo}
\end{align}
and we shall study successively the six domains in the right-hand side of \eqref{decompo}.

The function $r_y$ cannot have poles in the first domain, since the
latter is equal to $\Delta_y$ (Figure
\ref{The_fundamental_parallelogram}), where $r_y$ is defined through
its GF, see \eqref{rxryrxry}. In the second domain, $r_y$ is defined
thanks to $r_y(\o)= -r_x(\o)+K(0,0)Q(0,0)+x(\o)y(\o)$, see
\eqref{rDrD}. The second domain being included in $\Delta_x$ (Figure
\ref{The_fundamental_parallelogram}), the function $r_x$ has no
poles there, and the possible singularities of $r_y$ necessarily
come from the term $x(\o)y(\o)$. Using Lemma \ref{lem:lemKRG}, we
find only one pole in that domain, namely at $3\o_2/8$, of order $1$.
  To compute its residue we notice that the function $x(\o)y(\o)$
 has the same principal part as the function $-f_y(\o)$
at $3\o_2/8$ due to the expression \eqref{fxfy}: in fact, by
Lemma~\ref{lem:lemKRG}, the point $3\o_2/8$ is not a pole of $x(\o)$
and the point $-3\o_2/8+\o_1+\o_2$ is not a pole of $y(\o)$. By
Lemma~\ref{lemma:f_y_gessel} the point $3\o_2/8$ is a simple pole of $-f_y$
with residue $-1/(2z)$, and hence of $x(\o)y(\o)$ as well.
 We record this information
in Table \ref{table:poles} below.

 \begin{table}[ht]
 \begin{tabular}{| l | c | c | c | c | c | c | c | c |}
  \hline
  Point & $\o_2/8$ & $3\o_2/8$ & $11\o_2/8$ & $13\o_2/8$ & $15\o_2/8$ & $17\o_2/8$ & $21\o_2/8$ & $23\o_2/8$ \\
  \hline
  Residue&  $1/(2z)$ & $-1/(2z)$ & $1/(2z)$ &$-1/(2z)$  &$-1/(2z)$ &$1/z$ &$-1/z$ &
$1/(2z)$  \\
  \hline
\end{tabular}
\medskip
\caption{The points of the domain $\o_1[-1/2,1/2)+\o_2[1/8,25/8)$ where the function $r_y$ has poles and the residues at these poles}
\label{table:poles}
\end{table}

We now consider the third domain. We use the equation
\begin{equation}
\label{sdgh}
 r_y(\omega+6\o_2/8)=r_y(\omega)+f_y(\o),
 \end{equation} see
\eqref{cont1} together with (\ref{kn}). Since $r_y$ and $f_y$ both
have a pole at $3\o_2/8$ (see just above for $r_y$ and Lemma
\ref{lemma:f_y_gessel} for $f_y$), $r_y$ has a priori a pole at
$9\o_2/8$. The residue is the sum of residues of $r_y$ and $f_y$ at
$3\o_2/8$: $-1/(2z)+1/(2z)=0$,
 so that the singularity $9\o_2/8$ is removable.
 The point $3\o_2/8$ is the unique pole of $r_y$ on
 $\o_1[-1/2,1/2)+\o_2[3/8,9/8)$ by the previous analysis.
 It follows that, except for $9\o_2/8$,
 the poles of $r_y$ on the third domain
  ${\o_1}[-1/2,1/2)+{\o_2}[9/8,15/8)$
 necessarily arise by (\ref{sdgh}) from those of $f_y$
 on $\o_1[-1/2,1/2)+\o_2[3/8,9/8)$, that is $5\o_2/8$ and
 $7\o_2/8$.
 Lemma \ref{lemma:f_y_gessel}
 thus implies that $r_y$ has poles at $11\o_2/8$ and $13\o_2/8$,
  with respective residues $1/(2z)$ and $-1/(2z)$.
   These results are summarized in Table \ref{table:poles}.
 
   For the fourth and the fifth domains, we use exactly the same arguments,
   namely  equation \eqref{sdgh} and the knowledge of the poles
   in the previous domains.
   For the fourth domain ${\o_1}[-1/2,1/2)+{\o_2}[15/8,21/8)$,
     by (\ref{sdgh}), the poles of $r_y$ can arise from
   $9\o_2/8$ where $f_y$ has a pole, and from $11\o_2/8$, $13\o_2/8$ were
     both $r_y$ and $f_y$ have poles.
    Then the residue at $15\o_2/8$ is $-1/(2z)$,
the one at $17\o_2/8$ is $1/(2z)+ 1/(2z)=1/z$, the residue at
$19\o_2/8$ is $-1/(2z)+ 1/(2z)=0$, so that $19\o_2/8$ is a removable
singularity. For the fifth domain
${\o_1}[-1/2,1/2)+{\o_2}[21/8,25/8)$, by (\ref{sdgh}), the poles may
come from $15\o_2/8$ and $17\o_2/8$ where both $r_y$ and $f_y$ have
poles. The residue at $21\o_2/8$ is $-1/(2z)-1/(2z)=-1/z$, the one
at  $23\o_2/8$ is $1/z-1/(2z)=1/(2z)$.

As for the last domain, we can use equation  (\ref{sdgh})  under the
form $r_y(\omega)=r_y(\omega+6\o_2/8)-f_y(\o)$.
   As already proven, $r_y$ has no poles at $[7\o_2/8, \o_2)$,
   hence, the only poles of $r_y$ in this domain are those of
   $-f_y$. By Lemma \ref{lemma:f_y_gessel} this is $\o_2/8$ with the
   residue $1/(2z)$.

 The proof of Table
\ref{table:poles} is complete.

To conclude the proof of Theorem \ref{thm:Gessel-bi} we use Property \ref{expression_elliptic_zeta} of Appendix \ref{appendix-elliptic}. This property allows to express $r_y$ as a sum of a constant $c$ and of eight $\zeta$-functions (eight because there are eight poles in Table \ref{table:poles}), exactly as in equation \eqref{zeta}.
\end{proof}

\begin{proof}[Proof of Theorem \ref{thm:first_Gessel_simplified}]
Equation \eqref{sqs2} yields $r_x(\o) =
x(\o)y(\o)-r_y(\o)+K(0,0)Q(0,0)$.
   We compute the constant $K(0,0)Q(0,0)$  as $r_y(\o_0^y)=K(0,y(\o_0^y ))Q(0, y( \o_0^y ))$,
 where $\o_0^y \in \Delta_y$ is such that $y(\o_0^y)=0$.
 Lemma \ref{lem:lemKRG} gives a unique possibility for $\o_0^y$, namely, $\o_0^y=7\o_2/8$.
   Hence $r_x(\o)=
x(\o)y(\o)-r_y(\o)+r_y(7\o_2/8)$.
  Let us substitute $\o=5\o_2/8$ in this equation. The point $ 5\o_2/8$
   is a zero of $x(\o)$
    that lies in $\Delta_x$, so that
    \begin{equation*}
     r_x(5\o_2/8)= K(x(5\o_2/8
    ),0)Q(x(5\o_2/8),0)=K(0,0) Q(0,0)=z Q(0,0).
    \end{equation*}
        This point is not a pole of $y(\o)$, in such a way that
      $x(5\o_2/8 )y(5\o_2/8 )=0$. We obtain
\begin{equation}
\label{eq:first_Gessel}
     zQ(0,0) = r_y(7\o_2/8)-r_y(5\o_2/8).
\end{equation}
Note in particular that in order to obtain the expression
\eqref{eq:first_Gessel} of $Q(0,0)$,
 there is no need to know the constant $c$ in Theorem \ref{thm:Gessel-bi}.

With Theorem \ref{thm:Gessel-bi} and \eqref{eq:first_Gessel},
$Q(0,0)$ can be written as a sum of $16$ Weierstrass $\zeta_{1,3}$-functions
 (each of them being evaluated at a rational multiple of $\o_2$).
 Using the fact that $\zeta_{1,3}$ is an odd function and using property \ref{half_period_translated_zeta},
  we can perform many easy simplifications in \eqref{eq:first_Gessel}, and, this way, we obtain \eqref{eq:first_Gessel_simplified}.
\end{proof}

We shall see in Section \ref{sec:proof_B} how to deduce from Theorem \ref{thm:Gessel-bi} the expression of $Q(0,y;z)$ (and in fact, the expression of all its algebraic branches).


\section{Proof of Gessel's conjecture (Problem \ref{enumi:problem_A})}
\label{sec:proof_A}

In this section, we shall prove Gessel's formula~\eqref{eq:Gessel's_conjecture} for the number of Gessel excursions.
  The starting point is Theorem~\ref{thm:first_Gessel_simplified}, which expresses the generating function of Gessel excursions as a linear combination of (evaluations at multiples of $\o_2/4$ of) the Weierstrass zeta function $\zeta_{1,3}$ with periods $\o_1,3\o_2$. The individual terms of this linear combination are (possibly) transcendental functions; our strategy is to group them in a way that brings up a linear combination of algebraic hypergeometric functions, from which Gessel's conjecture follows by telescopic summation.

\subsection*{Roadmap of the proof}
More precisely, Gessel's formula~\eqref{eq:Gessel's_conjecture} is
equivalent to\footnote{This was already pointed out by Ira Gessel when he initially formulated the conjecture.} 
\begin{equation*}
     Q(0,0;z) = \frac{1}{2z^2} \left( {}_2F_1\left(\left[-\frac{1}{2},-\frac{1}{6}\right],\left[\frac{2}{3}\right],16z^2\right)-1 \right).
\end{equation*}
Here, we use the notation ${}_2F_1([a,b],[c],z)$ for the Gaussian hypergeometric function
\begin{equation}
\label{eq:def_hypergeometric_functions}
     {}_2F_1([a,b],[c],z) = \sum_{n=0}^{\infty}\frac{(a)_n\cdot (b)_n}{(c)_n}\frac{z^n}{n!}.
\end{equation}

In view of Theorem~\ref{thm:first_Gessel_simplified}, Gessel's conjecture is therefore equivalent to
\begin{equation}\label{eq:equivGessel}
   L_1 - 3L_2 + 2L_3 + 3L_4 -5L_5 + 2L_6 = G-1,
\end{equation}
where $G=G(z)$ is the algebraic hypergeometric function ${}_2F_1([-1/2,-1/6],[2/3],16z^2)$, and where $L_k$ denotes the function $\zeta_{1,3}(k  \omega_2/4)$ for $1\leq k \leq 6$.

Let us denote by $V_{i,j,k}$ the function $L_i + L_j - L_k$.
Then, the left-hand side of equality~\eqref{eq:equivGessel} rewrites
$4V_{1,4,5} - V_{2,4,6} - V_{1,5,6} - 2V_{1,2,3}$.

To prove~\eqref{eq:equivGessel}, our key argument is encapsulated in the following identities:
\begin{align}
     V_{1,4,5} &= (2G+H)/3 - K/2,\label{eq:key145}\\
     V_{2,4,6} &= (2G+H)/3 - K, \label{eq:key246}\\
     V_{1,5,6} &= (J+1)/2,  \label{eq:key156}\\
     V_{1,2,3} &= (2G+2H-J-2K+1)/4. \label{eq:key123}
\end{align}
Here $H,J$ and $K$ are auxiliary algebraic functions, defined in the following way: $H$ is the hypergeometric function ${}_2F_1([-1/2,1/6],[1/3],16z^2)$, and $J$ stands for $(G-K)^2$, where $K$ is equal to $zG' = 4z^2{}_2F_1([1/2, 5/6],[5/3],16z^2)$.

Gessel's conjecture is a consequence of the equalities~\eqref{eq:key145}--\eqref{eq:key123}.
Indeed, by summation, these equalities imply that $4V_{1,4,5} - V_{2,4,6} - V_{1,5,6} - 2V_{1,2,3}$ is equal to $G-1$, proving~\eqref{eq:equivGessel}.

It then remains to prove equalities~\eqref{eq:key145}--\eqref{eq:key123}. To do this, we use the following strategy. Instead of proving the equalities of functions of the variable $z$, we rather prove that their evaluations at
$z=(x(x+1)^3/(4x+1)^3)^{1/2}$ are equal. This is sufficient, since the map 
$\varphi : x \mapsto (x(x+1)^3/(4x+1)^3)^{1/2}$ is a diffeomorphism between $(0,1/2)$ onto $(0,1/4)$.
The choice of this algebraic transformation is inspired by the Darboux covering for tetrahedral hypergeometric equations of the Schwarz type $(1/3, 1/3, 2/3)$~\cite[\S6.1]{Vidunas08}. 

First, we make use of a corollary of the Frobenius-Stickelberger
identity~\ref{Frobenius} (\cite[page 446]{WW}), which implies that $V_{i,j,k}$
is equal to the algebraic function $\sqrt{T_i + T_j + T_k}$ as soon as
$k=i+j$. Here, $T_\ell$ denotes the algebraic function $\wp_{1,3}(\ell
\omega_2/4)$. Second, using classical properties of the Weierstrass
functions~$\wp$ and~$\zeta$, we explicitly determine $T_\ell(\varphi(x))$ for
$1\leq \ell \leq 6$, then use them to compute $V_{1,4,5},V_{2,4,6},V_{1,5,6}$
and $V_{1,2,3}$ evaluated at $z=\varphi(x)$. Finally,
equalities~\eqref{eq:key145}--\eqref{eq:key123} are proved by checking that
they hold when evaluated at $z=\varphi(x)$.

\subsection*{Preliminary results}
We shall deal with elliptic functions with different
pairs of periods. We shall denote by $\zeta,\wp$ the elliptic
functions with periods $\o_1,\o_2$, and by $\zeta_{1,3},\wp_{1,3}$
the elliptic functions with periods $\o_1,3\o_2$, see Appendix
\ref{appendix-elliptic} for their definition and properties. Further, we recall
that elliptic functions are alternatively characterized by their
periods (see equation \eqref{eq:first_time_expansion_wp}) or by
their invariants. The invariants of $\wp$ are denoted by $g_2,g_3$.
They are such that
\begin{equation}
\label{eq:diff_eq}
     \wp'(\o)^2=4\wp(\o)^3-g_2\wp(\o)-g_3,\qquad \forall \o\in{\bf C}.
\end{equation}
We recall from \cite[Lemma 12]{KRG} the following result that provides explicit expressions for the invariants $g_2,g_3$.
\begin{lem}[\cite{KRG}]
\label{lemma:invariants_o_1_o_2-Gessel}
We have
\begin{equation}
\label{eq:invariants_o_1_o_2-Gessel}
     g_2 = (4/3)(1-16z^2+16z^4),\qquad
     g_3=-(8/27)(1-8z^2)(1-16z^2-8z^4).
\end{equation}
\end{lem}
Likewise, we define the invariants $g_2^{1,3},g_3^{1,3}$ of $\wp_{1,3}$.
To compute them, it is convenient to first introduce an algebraic function denoted~$R$, which is the unique positive root of
\begin{equation}
\label{eq:def_R_Gessel}
     X^4-2g_2X^2+8g_3X-g_2^2/3=0.
\end{equation}
To prove that \eqref{eq:def_R_Gessel} has a unique positive root, we need to introduce the
 discriminant of the fourth-degree polynomial $P(X)$ defined by \eqref{eq:def_R_Gessel}.
 Since $\deg P(X)=4$ and since the leading coefficient of $P(X)$ is $1$, its discriminant equals
 the resultant of $P(X)$ and $P'(X)$. Some elementary computations give that it equals $cz^{16}(1-16z^2)^2$, where $c$ is a negative constant.
The discriminant is thus negative (for any $z\in(0,1/4)$). On the other hand, the discriminant can be interpreted as $\prod_{1\leq i<j\leq 4}(R_i-R_j)^2$, where the $R_i$, $i\in\{1,\ldots ,4\}$, are the roots of $P(X)$. The negative sign of the discriminant implies that $P(X)$ has two complex conjugate roots and two real roots. Further, the product of the roots is clearly negative, see \eqref{eq:def_R_Gessel}, so that one of the two real roots is negative while another one is positive.

Using equations \eqref{eq:invariants_o_1_o_2-Gessel} and \eqref{eq:def_R_Gessel}, we obtain the local expansion $R(z) = 2-16z^2-48z^4+O(z^6)$ in the neighborhood of $0$.

The algebraic function $R$ will play an important role in determining the algebraic functions $T_\ell = \wp_{1,3}(\ell  \omega_2/4)$. 
To begin with, the next lemma expresses $T_4$, as well as the invariants $g_2^{1,3},g_3^{1,3}$, in terms of $R$.

\begin{lem}
\label{lem:three_important_values}
One has
\begin{align*}
     T_4 &= \wp_{1,3}(\o_2) = R/6,\\
     g_2^{1,3} &= -g_2/9+10R^2/27,\\
     g_3^{1,3} & = -35R^3/729+7g_2R/243-g_3/27,
\end{align*}
  where expressions for $g_2$ and $g_3$ are given in~\eqref{eq:invariants_o_1_o_2-Gessel}.
\end{lem}
\begin{proof}
Using the properties \ref{addition_theorem} and \ref{principle_transformation} from Lemma~\ref{lemma_properties_wp},
 one can write,
\begin{equation}
\label{eq:after_two_at}
     \wp(\o) = -4\wp_{1,3}(\o_2)-\wp_{1,3}(\o)+\frac{\wp_{1,3}'(\o)^2+\wp_{1,3}'(\o_2)^2}{2(\wp_{1,3}(\o)-\wp_{1,3}(\o_2))^2},\qquad \forall\o\in{\bf C}.
\end{equation}
We then make a local expansion at the origin of the both sides of the equation above, using property~\ref{expansion_0} from Lemma~\ref{lemma_properties_wp}. We obtain
\begin{align*}
     &\frac{1}{\o^2}+\frac{g_2}{20}\o^2+\frac{g_3}{28}\o^4+O(\o^6)=\\
     &\frac{1}{\o^2}+\left(6\wp_{1,3}(\o_2)^2-\frac{9g_2^{1,3}}{20}\right)\o^2+\left(10\wp_{1,3}(\o_2)^3-\frac{3g_2^{1,3}\wp_{1,3}(\o_2)}{2}-\frac{27g_3^{1,3}}{28}\right)\o^4+O(\o^6).
\end{align*}
Identifying the expansions above, we obtain two equations for the three unknowns
$\wp_{1,3}(\o_2)$, $g_2^{1,3}$ and $g_3^{1,3}$
(remember that its invariants $g_2$ and $g_3$ are known from Lemma \ref{lemma:invariants_o_1_o_2-Gessel}).
 We add a third equation by noticing that $\wp_{1,3}(\o_2)$
  is the only real positive solution to (see, e.g., \cite[Proof of Lemma~22]{KRG})
\begin{equation*}
     X^4-\frac{g_2^{1,3}}{2}X^2-g_{3}^{1,3}X-\frac{(g_2^{1,3})^2}{48}=0.
\end{equation*}
We then have a (non-linear) system of three equations with three unknowns.
A few easy computations finally lead to the expressions of $\wp_{1,3}(\o_2)$,
  $g_2^{1,3}$ and $g_3^{1,3}$ of Lemma \ref{lem:three_important_values}.
\end{proof}

The next result expresses the algebraic functions $T_1, T_2, T_3, T_5$ and $T_6$ in terms of the algebraic function $R$, and of the invariants $g_2$ and $g_3$ (the quantity $T_4$ has already been found in Lemma \ref{lem:three_important_values}).
\begin{lem}
\label{lem:three_formulae}
One has the following formul\ae:
\begin{enumerate}
     \item\label{item:formula_0} $T_1 = \wp_{1,3}(\o_2/4)$ is the unique solution of
\begin{multline}
    \label{bbb}
     X^3-\left(\frac{R}{3}+\frac{1+4z^2}{3}\right)X^2+\left(\frac{R
    (1+4z^2)}{9}+\frac{R^2}{108}+\frac{g_2}{18}\right)X\\
     +\left(\frac{23R^3}{2916}-\frac{R^2
    (1+4z^2)}{108}+\frac{g_3}{27}-\frac{19Rg_2}{972}\right)=0
\end{multline}
such that in the neighborhood of $0$, $T_1 =1/3+4z^2/3-4z^6-56z^8+O(z^{10})$.
     \item\label{item:formula_1} $T_2 = \wp_{1,3}(2\o_2/4)$ is equal to
\begin{equation}
\label{eq-Gessel-definition-T_2}
     T_2 = \frac{R+1-8z^2}{6}-{\frac{T_6}{2}}.
\end{equation}
    \item\label{item:formula_0_ter} $T_3 = \wp_{1,3}(3\o_2/4)$ is the unique solution of \eqref{bbb} such that in the neighborhood~of~$0$, $T_3 = 1/3-8z^2/3-8z^4-60z^6+O(z^8)$.
          \item\label{item:formula_0_bis} $T_5 = \wp_{1,3}(5\o_2/4)$ is the unique solution of \eqref{bbb} such that in the neighborhood~of~$0$, $T_5=1/3-8z^2/3-8z^4-64z^6+O(z^8)$.
              \item\label{item:formula_2} $T_6 = \wp_{1,3}(6\o_2/4)$ is equal to
    \begin{equation}
    \label{eq-Gessel-definition-T_6}
         T_6 = \frac{R+1-8z^2-\sqrt{3R^2-4R(1-8z^2)+4(1-8z^2)^2-6g_2}}{9}.
    \end{equation}
\end{enumerate}
\end{lem}

\begin{proof}
We first prove that for a given value of $\o$ (and thus for a given value of $\wp(\o)$),
the three solutions of
\begin{multline}
\label{xxxy}
X^3-\left(\frac{R}{3}+\wp(\o)\right)X^2+\left(\frac{R\wp(\o)}{3}+\frac{R^2}{108}
+\frac{g_2}{18}\right)X\\+\left(\frac{23R^3}{2916}-\frac{\wp(\o)R^2}{36}+\frac{g_3}{27}-\frac{19Rg_2}{972}\right)=0
\end{multline}
 are
\begin{equation*}
     \{\wp_{1,3}(\o), \wp_{1,3}(\o+\o_2), \wp_{1,3}(\o+2\o_2)\}.
\end{equation*}
 By property
\ref{principle_transformation} we find
\begin{equation*}
     \wp(\o) = - 4\wp_{1,3}(\o_2)-\wp_{1,3}(\o)+\frac{\wp'_{1,3}(\o)^2+\wp'_{1,3}(\o_2)^2}{2(\wp_{1,3}(\o)-\wp_{1,3}(\o_2))^2},\qquad \forall \o\in{\bf C},
\end{equation*}
where by Lemma \ref{lem:three_important_values}, one has $\wp_{1,3}(\o_2)=R/6$.
Then $\wp'_{1,3}(\o_2)^2=4(R/6)^3-g_2^{1,3}R/6-g_3^{1,3}$, and following this way, we obtain that $\wp_{1,3}(\o)$ satisfies \eqref{xxxy}.

We start the proof of the lemma by showing \ref{item:formula_0}.
Using \cite[Lemma 19]{KRG} one has that
$\wp(\o_2/4)=\wp(3\o_2/4)=(1+4z^2)/3$. Then the equation \eqref{bbb}
is exactly \eqref{xxxy} with $\o=\o_2/4$.
 The three roots of \eqref{bbb} are $\wp_{1,3}(\o_2/4)$,
 $\wp_{1,3}(5\o_2/4)$ and $\wp_{1,3}(9\o_2/4)=\wp_{1,3}(3\o_2/4)$.
  By using standard properties of the Weierstrass function $\wp$, we see that
  $\wp_{1,3}(\o_2/4)$ is the largest of the three quantities
  (and this for any $z\in(0,1/4)$).
  Further, since \eqref{xxxy} is a polynomial of degree $3$,
  we can easily find its roots in terms of the variable $z$.
  This way, we find that the three solutions admit the expansions
  \begin{align*}
  &1/3+4z^2/3-4z^4-56z^6+O(z^8),\\
  &1/3-8z^2/3-8z^4-60z^6+O(z^8),\\
  &1/3-8z^2/3-8z^4-64z^6+O(z^8).
  \end{align*}
  Accordingly,  $T_1=\wp_{1,3}(\o_2/4)$ corresponds to the first one, $T_3=\wp_{1,3}(3\o_2/4)$ to the second one and $T_5=\wp_{1,3}(5\o_2/4)$ to the last one.

We now prove \ref{item:formula_1} and \ref{item:formula_2}. Using
again \cite[Lemma 19]{KRG}, one derives that $\wp(2\o_2/4)
=(1-8z^2)/3$. The three roots of equation \eqref{xxxy} with
$\o=2\o_2/4$ are $\wp_{1,3}(2\o_2/4)$, $\wp_{1,3}(6\o_2/4)$ and
$\wp_{1,3}(10\o_2/4)$. Since
$\wp_{1,3}(10\o_2/4)=\wp_{1,3}(2\o_2/4)$, \eqref{xxxy} with
$\o=2\o_2/4$ has a double root (that we call $t_1$) and a simple
root ($t_2$). It happens to be simpler to deal now with the
derivative of the polynomial in the left-hand side of \eqref{xxxy}.
It is an easy exercise to show that the roots of the derivative of a polynomial of degree $3$ with a double
root at $t_1$ and a simple root at $t_2$
are $t_1$ and $(t_1+2t_2)/3$. This way, we obtain expressions for
$\wp_{1,3}(2\o_2/4)$ and
$(\wp_{1,3}(2\o_2/4)+2\wp_{1,3}(6\o_2/4))/3$, which are equal to
\begin{equation}
\label{eq:before_sign}
     \frac{R+1-8z^2\pm\sqrt{3R^2-4R(1-8z^2)+4(1-8z^2)^2-6g_2}}{9}.
\end{equation}
Since $\wp_{1,3}(2\o_2/4)>\wp_{1,3}(6\o_2/4)$, the root $\wp_{1,3}(2\o_2/4)$
corresponds to the sign $+$ in \eqref{eq:before_sign}. This way we
immediately find expressions for $\wp_{1,3}(2\o_2/4)$ and
$\wp_{1,3}(6\o_2/4)$, and this finishes the proof of the lemma.
\end{proof}

Let $\varphi: (0,1/2) \rightarrow (0,1/4)$ be the diffeomorphism defined by $\varphi(x) = \sqrt{{x(x+1)^3}/{(4x+1)^3}}$. 
The next result derives  explicit expressions of $T_\ell(z)$ when evaluated at
$z=\varphi(x)$.

\begin{lem}
	Define \[M(x) = \frac{4x^4+28x^3+30x^2+10x+1}{3(4x+1)^3} \quad \text{and} \quad N(x) = \frac{2x(x+1)(2x+1)}{(4x+1)^{5/2}}.\]
For any $x\in (0,1/2)$, the following formul\ae\ hold:
\begin{align*}
T_1 \left( \varphi(x) \right) &= M(x) + N(x), \\
T_2 \left( \varphi(x) \right) &= M(x) - \frac{2x(x+1)(2x+1)}{(4x+1)^3}, \\
T_3 \left( \varphi(x) \right) &= M(x) - \frac{2x(x+1)}{(4x+1)^2}, \\
T_4 \left( \varphi(x) \right) &= M(x) - \frac{2x(2x+1)(3x+1)}{(4x+1)^3}, \\
T_5 \left( \varphi(x) \right) &= M(x) - N(x),\\
T_6 \left( \varphi(x) \right) &= \left(\frac{2x+1}{4x+1}\right)^2 - 2 M(x).
\end{align*}
\end{lem}

\begin{proof}
    All equalities are consequences of Lemma~\ref{lem:three_formulae}.
    We begin with $R$: we replace $z$ by $\varphi(x)$ in equation~\eqref{eq:def_R_Gessel}, factor the result, and identify the corresponding minimal polynomial of $R(\varphi(x))$ in ${\bf Q}(x)[T]$. To do this, we use that the local expansion of $R(\varphi(x))$ at $x=0$
is equal to $2-16x+O(x^2)$.
The minimal polynomial has degree 1, proving the equality
\[R \left( \varphi(x) \right) = \frac{2(2x^2-2x-1)^2}{(4x+1)^3}.\]
    {}From $R$, we directly deduce $T_4=R/6$.
    Now, replacing $z$ by $\varphi(x)$ in Lemma~\ref{lem:three_formulae} \ref{item:formula_2} provides the expression of $T_6(\varphi(x))$. Then $T_2$ is treated in a similar way using Lemma~\ref{lem:three_formulae} \ref{item:formula_1}. Finally, an annihilating polynomial for $T_1(\varphi(x)), T_3(\varphi(x)), T_5(\varphi(x))$ is deduced in a similar manner using Lemma~\ref{lem:three_formulae} \ref{item:formula_0}. This polynomial in ${\bf Q}(x)[T]$ factors as a product of a linear factor and a quadratic factor. Using the local expansions $1/3+4/3x-12x^2+80x^3+O(x^4), 1/3-8/3x+16x^2-84x^3+O(x^4)$ and $1/3-8/3x+16x^2-88x^3+O(x^4)$ allows to conclude.
\end{proof}
The following Corollary is a direct consequence of the previous lemma and of the Frobenius-Stickelberger identity~\ref{Frobenius}. 

\begin{cor}\label{V-in-x}
The algebraic functions $V_{1,4,5}, V_{2,4,6}, V_{1,5,6}$ and $V_{1,2,3}$ defined in~\eqref{eq:key145}--\eqref{eq:key123} satisfy the following equalities for any $x\in(0,1/2)$:
\begin{align*}
    V_{1,4,5} \left( \varphi(x) \right) &= \frac{2x^2+4x+1}{(4x+1)^{3/2}}, \\
    V_{2,4,6} \left( \varphi(x) \right) &= \frac{2x+1}{(4x+1)^{3/2}}, \\
    V_{1,5,6} \left( \varphi(x) \right)& = \frac{2x+1}{4x+1}, \\
    V_{1,2,3} \left( \varphi(x) \right) &= \frac{x}{4x+1}  + \frac{(x+1)(2x+1)}{(4x+1)^{3/2}}.
\end{align*}
\end{cor}

\subsection*{Completing the proof of Gessel's conjecture}

 The last step consists in proving the four equalities~\eqref{eq:key145}--\eqref{eq:key123}.
The starting point is that the hypergeometric power series
$G = {}_2F_1([-1/2,-1/6],[2/3],16z^2)$, $H = {}_2F_1([-1/2,1/6],[1/3],16z^2)$, $K=z \cdot G'$ and $J=(G-K)^2$ are algebraic and satisfy the equations displayed in the following lemma.

\begin{lem}\label{G-H-J-K-in-x}
For any $x\in(0,1/2)$, one has the following formul\ae:
\begin{align*}
G \left( \varphi(x) \right) &= \frac{4x^2+8x+1}{(4x+1)^{3/2}}, \\
H \left( \varphi(x) \right) &= \frac{4x^2+2x+1}{(4x+1)^{3/2}}, \\
K \left( \varphi(x) \right)& = \frac{4x(x+1)}{(4x+1)^{3/2}}, \\
J \left( \varphi(x) \right) &= \frac{1}{4x+1}.
\end{align*}
\end{lem}

\begin{proof}
The first two equalities are consequences of identities (46)--(49) in~\cite[Section 6.1]{Vidunas08}.
Let us prove the first equality. We start with the contiguity relation~\cite[Eq.~15.2.15]{AS} 
\[{}_2F_1([-1/2,-1/6],[2/3],z) = 2\cdot {}_2F_1([1/2,-1/6],[2/3],z) + (z-1)\cdot {}_2F_1([1/2,5/6],[2/3],z),\]
that we evaluate at $z=\psi(x) = 16\varphi(x)^2$. 
It follows that 
\[G(\varphi(x)) = 2\cdot {}_2F_1([1/2,-1/6],[2/3],\psi(x)) + (\psi(x)-1)\cdot {}_2F_1([1/2,5/6],[2/3],\psi(x)).\]
Identities (46)--(47) in~\cite[Section 6.1]{Vidunas08} 
write
\[{}_2F_1([1/2,-1/6],[2/3],\psi(x)) = (1+4x)^{-1/2},\quad \text{and}\]
\[{}_2F_1([1/2,5/6],[2/3],\psi(x)) = (1+4x)^{3/2}/(1-2x)^2.\]
Therefore, $G(\varphi(x))$ is equal to 
\[G(\varphi(x)) = 2\cdot (1+4x)^{-1/2} + (\psi(x)-1)\cdot \frac{(1+4x)^{3/2}}{(1-2x)^2} = \frac{4x^2+8x+1}{(1+4x)^{3/2}}.\]
To prove the second equality, we start from Euler's formula~\cite[Eq.~15.3.3]{AS} 
\[{}_2F_1([-1/2,1/6],[1/3], z) 
= (1-z)^{2/3} \cdot {}_2F_1([5/6,1/6],[1/3], z)\]
and from the contiguity relation~\cite[Eq.~15.2.25]{AS} 
\[ {}_2F_1([5/6,1/6],[1/3], z) =  \frac{z}{2z-2} \cdot {}_2F_1([5/6,1/6],[4/3], z) + \frac{1}{1-z} {}_2F_1([-1/6,1/6],[1/3], z).
\]
Putting them together, evaluating the result at $z=\psi(x)$ and  
using identities (48)--(49) in~\cite[Section 6.1]{Vidunas08} 
shows that $H(\varphi(x))$ is equal to
\[
 (1-\psi(x))^{2/3} \cdot
\left( \frac{\psi(x)}{2\psi(x)-2} \cdot \frac{(1+4x)^{1/2}(1+2x)^{1/3}}{1+x}
+ 
\frac{1}{1-\psi(x)} \cdot \frac{(1+2x)^{1/3}}{(1+4x)^{1/2}}\right),\]
which further simplifies to $(4x^2+2x+1)/(4x+1)^{3/2}$ for $x\in(0,1/2)$.

The last two equalities are easy consequences of the first two.
\end{proof}

Now, equalities~\eqref{eq:key145}--\eqref{eq:key123} evaluated at $z=\varphi(x)=(x(x+1)^3/(4x+1)^3)^{1/2}$ are easily proven using Lemmas~\ref{V-in-x} and~\ref{G-H-J-K-in-x}.
For instance, equality~\eqref{eq:key145} evaluated at $z=\varphi(x)$ reads:
\[
\frac{2x^2+4x+1}{(4x+1)^{3/2}} =
\frac23 \frac{4x^2+8x+1}{(4x+1)^{3/2}} +
\frac13 \frac{4x^2+2x+1}{(4x+1)^{3/2}} -
\frac12 \frac{4x(x+1)}{(4x+1)^{3/2}}.
\]
Similarly, equality~\eqref{eq:key123} evaluated at $z=\varphi(x)$ reads:
\[
\frac{x}{4x+1}  + \frac{(x+1)(2x+1)}{(4x+1)^{3/2}} =
\frac12 \frac{4x^2+8x+1}{(4x+1)^{3/2}} +
\frac12 \frac{4x^2+2x+1}{(4x+1)^{3/2}} -
\frac14 \frac{1}{4x+1} -
\frac12 \frac{4x(x+1)}{(4x+1)^{3/2}}+
\frac14.
\]

\medskip

The proof of Gessel's formula~\eqref{eq:Gessel's_conjecture} for the number of Gessel excursions is thus completed. Note that incidentally we have proved that the generating series $Q(0,0;z)$ for Gessel excursions is \emph{algebraic}. The next section is devoted to the proof that the \emph{complete\/} generating series of Gessel walks is also algebraic.

\section{Proof of the algebraicity of the GF (Problem \ref{enumi:problem_B})}
\label{sec:proof_B}

\subsection*{Branches of the GFs and algebraicity of $Q(x,y)$ in the variables $x,y$.}
  In this section we  prove a weakened version
of Problem \ref{enumi:problem_B}: we show that $Q(x,y)$ is algebraic
in $x,y$ (we shall prove in the next section that the latter function is algebraic
in $x,y,z$, which is much stronger). This is not necessary for our analysis, but this illustrates that our approach easily yields algebraicity results.

We first propose two proofs of the algebraicity of $Q(0,y)$ as a
function of $y$. The first proof is an immediate application of
property \ref{expression_elliptic_zeta_direct}. The sum of the
residues (i.e., the multiplicative factors in front of the $\zeta$-functions) in the formula \eqref{zeta} of Theorem \ref{thm:Gessel-bi}
is clearly $0$, so that
$r_y(\o)$ is an algebraic function of $\wp_{1,3}(\o)$, by \ref{expression_elliptic_zeta_direct}. Further, by
\ref{principle_transformation}, $\wp_{1,3}(\o)$ is an algebraic
function of $\wp(\o)$, and finally by
\eqref{eq:uniformization_Gessel}, $\wp(\o)$ is algebraic in $y(\o)$.
This eventually implies that $r_y(\o)$ is algebraic in $y(\o)$, or
equivalently that $Q(0,y)$ is algebraic in $y$, see \eqref{rxryrxry}.

The second proof is based on the meromorphic continuation of the GFs on the universal covering, which was recalled in Section \ref{sec:mero}. The restrictions of $r_y(\o)/K(0, y(\o))$ to the half parallelogram
     \begin{equation*}
          \mathcal D_{k,\ell}=\omega_3/2+\omega_1[\ell,\ell+1)+\omega_2(k/2,(k+1)/2]
     \end{equation*}
for $k,\ell \in {\bf Z}$ provide all branches of $Q(0,y)$ on ${\bf C}\setminus([y_1,y_2]\cup[y_3,y_4])$ as follows:
\begin{equation*}
          Q(0,y)=\{r_y(\o)/K(0,y(\o)):
           \o \text{ is the (unique) element of }\mathcal D_{k,\ell} \text{ such that } y(\o)=y\},
     \end{equation*}
see \cite[Section 5.2]{KRIHES} for more details. Due to the $\o_1$-periodicity of $r_y(\o)$ and $y(\o)$ (see \eqref{buzz} and
 \eqref{eq:uniformization_Gessel}, respectively),
 the restrictions of these functions on $\mathcal{D}_{k,\ell}$ do not depend on
 $\ell\in {\bf Z}$, and therefore determine the same branch as on $\mathcal{D}_{k,0}$
 for any $\ell$. Furthermore, due to equation \eqref{xietabad},
 the restrictions of $r_y(\o)/K(0,y(\o))$ on $\mathcal{D}_{-k+1,0}$ and on
 $\mathcal{D}_{k,0}$ lead to the same branches for any $k\in{\bf Z}$. Hence,
 the restrictions of $r_y(\o)/K(0,y(\o))$ on $\mathcal{D}_{k,0}$ with $k\geq 1$
 provide all different branches of this function. In addition, Lemma \ref{lem:elliptic} says that
 $r_y$ is $3\o_2$-periodic. This fact yields that $Q(0,y)$ has (at most) six branches,
 and is thus algebraic.

     An analogous statement holds for (the restrictions of) the function
     $r_x(\o)/K(x(\o),0)$ and then for $Q(x,0)$.
  Using the functional equation \eqref{functional_equation}, we conclude that $Q(x,y)$ is algebraic
  in the two variables $x,y$.

In the next section, we refine the previous statement,
and prove that $Q(x,y)$ is algebraic in $x,y,z$ (Problem \ref{enumi:problem_B}).

\subsection*{Proof of the algebraicity of the complete GF}

  We start by proving the algebraicity of $Q(0,y)$ as a function of
  $y,z$. We consider the representation of
  $r_y(\o)$ given in Theorem \ref{thm:Gessel-bi} and apply eight times the addition theorem
\ref{addition_theorem} for $\zeta$-functions, namely
\begin{equation*}
     \zeta_{1,3}(\o-k\o_2/8) = \zeta_{1,3}(\o)-\zeta_{1,3}(k\o_2/8)+\frac{1}{2}\frac{\wp_{1,3}'(\o)+\wp_{1,3}'(k\o_2/8)}{\wp_{1,3}(\o)-\wp_{1,3}(k\o_2/8)},
\end{equation*}
for suitable values of $k\in{\bf Z}$ that can be deduced from \eqref{zeta}:
\begin{equation}
\label{eq:values_k}
     k\in\{1,3,11,13,15,17,21,23\}.
\end{equation}
We then make the weighted sum of the eight identities above (corresponding to the appropriate values of $k$ in \eqref{zeta}); this way, we obtain
\begin{equation*}
     r_y(\o)=U_1(\o)+U_2+U_3(\o),
\end{equation*}
where $U_1(\o)$ is the weighted sum of the eight functions $\zeta_{1,3}(\o)$, $U_2$ is the sum of $c$ and of the weighted sum of the eight quantities $-\zeta_{1,3}(k\o_2/8)$, and $U_3(\o)$ is the weighted sum of the eight quantities
\begin{equation}
\label{eq:wseq}
     \frac{1}{2}\frac{\wp_{1,3}'(\o)+\wp_{1,3}'(k\o_2/8)}{\wp_{1,3}(\o)-\wp_{1,3}(k\o_2/8)}.
\end{equation}

Since the sum of the residues in the formula \eqref{zeta} equals $0$,
the coefficient in front of $\zeta_{1,3}(\o)$ is $0$, so that $U_1(\o)$ is identically zero.
To prove that $U_2$ is algebraic in $z$, it suffices to use similar
arguments as we did to prove that $Q(0,0)$ is algebraic (we group together different $\zeta$-functions and we use standard identities as the Frobenius-Stickelberger equality \ref{Frobenius} or the addition formula for the $\zeta$-function \ref{addition_theorem}, see Section \ref{sec:proof_A}); we do not repeat the arguments here.
Finally, we show that $U_3(\o)$ is algebraic in $y(\o)$ over the
field of algebraic functions in $z$. In other words, we show that
there exists a non-zero bivariate polynomial $P$ such that
\begin{equation*}
P(U_3(\o),y(\o))=0,
\end{equation*}
where the coefficients of $P$ are algebraic functions in
$z$. This is enough to conclude the algebraicity of $Q(0,y)$ as a
function of $y$ and $z$.

To prove the latter fact, we shall prove that each term \eqref{eq:wseq} for
$k$ as in \eqref{eq:values_k} satisfies the property above (with different
polynomials $P$, of course). First, Lemma \ref{lem:algebraic_8} below implies
that $\wp_{1,3}(k\o_2/8)$ and $\wp_{1,3}'(k\o_2/8)$ are both algebraic in $z$.
Further, it follows from \ref{principle_transformation} that the function
$\wp_{1,3}(\o)$ is algebraic in $\wp(\o)$ over the field of algebraic
functions in $z$. It is thus also algebraic (over the field of algebraic
functions in $z$) in $y(\o)$, thanks to \eqref{eq:uniformization_Gessel}. The
same property holds for $\wp_{1,3}'(\o)$: this comes from the differential
equation \eqref{eq:diff_eq} satisfied by the Weierstrass elliptic functions.

The proof of the algebraicity of $Q(x,0)$ as a function of $x$ and $z$ is
analogous. With equation \eqref{functional_equation} the algebraicity of
$Q(x,y)$ as a function of $x,y,z$ is proved.\hfill$\square$

\smallskip To conclude this section it remains to prove the following lemma.
\begin{lem}
\label{lem:algebraic_8}
     For any $k\in{\bf Z}$ and any $\ell\in{\bf N}$, $\wp^{(\ell)}(k\o_2/8)$ and $\wp_{1,3}^{(\ell)}(k\o_2/8)$ are (infinite or) algebraic functions of $z$.
\end{lem}

\begin{proof}
First, for any $\ell\in{\bf N}$ and $k\in 8{\bf Z}$ (resp.\ $k\in 24{\bf Z}$), $\wp^{(\ell)}(k\o_2/8)$ (resp.\ $\wp_{1,3}^{(\ell)}(k\o_2/8)$) is infinite. For other values of $k$, they are finite. By periodicity and parity, it is enough to prove the algebraicity for $k\in\{1,\ldots ,4\}$ (resp.\ $k\in\{1,\ldots ,12\}$).

It is important to notice that it suffices to prove the result for $\ell=0$. Indeed, all the invariants $g_2$, $g_3$, $g_2^{1,3}$ and $g_3^{1,3}$ are algebraic functions of $z$ (see Lemmas \ref{lemma:invariants_o_1_o_2-Gessel} and \ref{lem:three_important_values}), so that using inductively the differential equation \eqref{eq:diff_eq}, we obtain the algebraicity for values of $\ell\geq 1$ from the algebraicity for $\ell=0$.

We first consider $\wp=\wp^{(0)}$. It is demonstrated in \cite[Lemma 19]{KRG} that
$\wp(k\o_2/8)$ is algebraic for $k=2$ and $k=4$. For $k=1$
this follows from the bisection formula \ref{bisection} and from the case $k=2$
 (note that $\wp(\o_1/2)$, $\wp(\o_2/2)$ and $\wp((\o_1+\o_2)/2)$ are algebraic in $z$, see \cite[Lemma 12]{KRG}).
 For $k=3$, this is a consequence of the addition formula \ref{addition_theorem}.

We now deal with $\wp_{1,3}$. Let $k\in\{1,\ldots ,12\}$. Using
\eqref{eq:after_two_at}, we easily derive that $\wp_{1,3}(\o_0)$ is
algebraic in $z$ as soon as $\wp(\o_0)$ is algebraic in $z$: indeed,
in \eqref{eq:after_two_at} the functions $\wp_{1,3}(\o_2)$ and $\wp_{1,3}'(\o_2)$
are algebraic in $z$, as a consequence of Lemma
\ref{lem:three_important_values}. This remark applies to all $\o_0=k\o_2/8$,
except for $k=8$, since then $\wp(k\o_2/8)=\infty$. In fact, for $k=8$
the situation is also simple, as $\wp_{1,3}(k\o_2/8)=R/6$ (see Lemma
\ref{lem:three_important_values}) is already known to be algebraic.
\end{proof}

\section{Conclusion}
\label{sec:conclusion}

\subsection*{Application of our results to other end points}
In this article we have presented the first human proofs of Gessel conjecture (Problem \ref{enumi:problem_A}) and of the algebraicity of the complete GF counting Gessel walks (Problem \ref{enumi:problem_B}). We have deduced the closed-form expression \eqref{eq:Gessel's_conjecture} of the numbers of walks $q(0,0;n)$ from a new algebraic expression of the GF $\sum_{n\geq 0} q(0,0;n)z^{n}$.

With a very similar analysis, we could obtain an expression for the series $\sum_{n\geq 0} q(i,j;n)z^{n}$, for any fixed couple $(i,j)\in{\bf N}^2$. Let us illustrate this fact with the example $(i,j)=(0,j)$, $j\geq0$. Let
\begin{equation*}
     g_j(z) = \sum_{n\geq 0} q(0,j;2n)z^{2n}
\end{equation*}
be the function counting walks ending at the point $(0,j)$ of the vertical axis. We obviously have
\begin{equation*}
     Q(0,y) = \sum_{j\geq 0} y^j g_j(z).
\end{equation*}
In particular, up to a constant factor $j!$, the functions $g_j(z)$ are exactly the successive derivatives of $Q(0,y)$
w.r.t.\ the variable $y$ evaluated at $y=0$. First, $g_0(z)=Q(0,0)$.
Further, one has $r_y(\o)=z(y(\o)+1)Q(0,y(\o))$.
 Differentiating w.r.t.\ $\o$ and evaluating at $\o_0^y$ (which is such that $y(\o_0^y)=0$), we find
\begin{equation*}
     g_1(z) = \frac{r_y'(\o_0^y)}{zy'(\o_0^y)}-Q(0,0).
\end{equation*}
All quantities above can be computed, and a similar analysis as in Section \ref{sec:proof_A} could lead to a closed-form expression for the GF of the numbers of walks $q(0,1;2n)$, $n\geq0$. Similarly, one could compute $g_2(z)$, $g_3(z)$, etc.

\subsection*{New Gessel conjectures}
For any $j\geq 0$, define
\begin{equation*}
     f_j(z) = (-1)^{j}(2j+1)z^{j}+2z^{j+1}\sum_{n\geq 0}q(0,j;2n)z^{n}.
\end{equation*}
With this notation, the closed-form expression \eqref{eq:Gessel's_conjecture} for the $q(0,0;2n)$ is equivalent to (see \cite{BK2,Gessel13})
\begin{equation}
\label{eq:Gessel_equiv}
     f_0\left(z\frac{(1+z)^3}{(1+4z)^3}\right)=\frac{1+8z+4z^2}{(1+4z)^{3/2}}.
\end{equation}
On March 2013, Ira Gessel \cite{Gessel13} proposed the following new conjectures: for any $j\geq 1$,
\begin{equation}
\label{eq:new_conjectures}
     f_j\left(z\frac{(1+z)^3}{(1+4z)^3}\right)=(-z)^j\frac{p_j(z)}{(1+4z)^{3/2+3j}},
\end{equation}
where $p_j(z)$ is a polynomial of degree $3j+2$ with positive coefficients (due to \eqref{eq:Gessel_equiv}, these new conjectures generalize the original one).

We shall not prove these conjectures in the present article. However, we do think that following the method sketched in the first part of Section \ref{sec:conclusion}, it could be possible to prove them, at least for small values of $j\geq 0$.

\appendix

\section{Some properties of elliptic functions}
\label{appendix-elliptic}

In this appendix, we bring together useful results on the Weierstrass functions~$\wp$ and~$\zeta$.
First, they are defined as follows for all $\omega\in{\bf C}$:
 \begin{align}
 \label{eq:first_time_expansion_wp}
 \wp(\o) = \frac{1}{\o^2}+\sum_{(\overline{n},\widehat{n})\in{\bf Z}^2\setminus \{(0,0)\}}\left( \frac{1}{(\o-(\overline{n}\,\overline{\o}+\widehat{n}\widehat{\o}))^2}-\frac{1}{(\overline{n}\,\overline{\o}+\widehat{n}\widehat{\o})^2}\right),
\\
      \zeta(\o) = \frac{1}{\o}+\sum_{(\overline{n},\widehat{n})\in{\bf Z}^2\setminus \{(0,0)\}}\left( \frac{1}{\o-(\overline{n}\, \overline{\o}+\widehat{n}\widehat{\o})}+\frac{1}{\overline{n}\,\overline{\o}+\widehat{n}\widehat{\o}}+\frac{\o}{(\overline{n}\,\overline{\o}+\widehat{n}\widehat{\o})^2}\right).
 \end{align}
\
The next lemma displays a collection of useful 
properties of the functions $\wp$ and $\zeta$. These properties are very classical; they can be found in \cite{AS,JS,WW} (see also the references in \cite{KRnew}).

\begin{lem}
\label{lemma_properties_wp} Let ${\zeta}$ and $\wp$ be the Weierstrass functions with periods $\overline{\omega},\widehat{\omega}$. Then:
\begin{enumerate}[label={\rm (P\arabic{*})},ref={\rm (P\arabic{*})}]
\item\label{expression_zeta}
	$\zeta$ (resp.\ $\wp$) has a unique pole in the fundamental parallelogram $[0,\overline{\omega})+[0,\widehat{\omega})$. It is of order $1$ (resp.\ $2$), at $0$, and has residue $1$ (resp.\ $0$, and principal part $1/\o^2$).
\item\label{elliptic_poles_0}An elliptic function with no poles in the fundamental parallelogram
$[0,\overline{\omega})+[0,\widehat{\omega})$ is constant.
\item\label{expansion_0}In the neighborhood of $0$, the function $\wp(\o)$ admits the expansion
\begin{equation*}
     \wp(\o) = \frac{1}{\o^2}+\frac{g_2}{20}\o^2+\frac{g_3}{28}\o^4+O(\o^6).
\end{equation*}
\item\label{addition_theorem}We have the addition theorems, for any $\omega,\widetilde{\omega}\in{\bf C}$:
     \begin{equation*}
          \zeta({\omega}+\widetilde{\omega})=\zeta({\omega})+\zeta(\widetilde{\omega})+\frac{1}{2}\frac{\wp'({\omega})-
          \wp'(\widetilde{\omega})}{\wp({\omega})-
          \wp(\widetilde{\omega})},
     \end{equation*}
     and
          \begin{equation*}
          \wp({\omega}+\widetilde{\omega})=-\wp({\omega})-\wp(\widetilde{\omega})+\frac{1}{4}\left(\frac{\wp'({\omega})-
          \wp'(\widetilde{\omega})}{\wp({\omega})-
          \wp(\widetilde{\omega})}\right)^2.
     \end{equation*}
\item\label{expression_elliptic_zeta_direct}For given $\widetilde{\o}_1,\ldots,\widetilde{\o}_p\in{\bf C}$, define
     \begin{equation}
     \label{eq:expression_elliptic_zeta}
          f(\o)=c+\sum_{1\leq \ell\leq p}r_\ell\zeta(\o-\widetilde{\o}_\ell),\qquad  \forall \omega\in{\bf C}.
     \end{equation}
The function $f$ above is elliptic if and only if $\sum_{1\leq \ell\leq p}r_\ell=0$.
\item\label{expression_elliptic_zeta}Let $f$ be an elliptic function with periods $\overline{\omega},\widehat{\omega}$ such that in the fundamental parallelogram $[0,\overline{\omega})+[0,\widehat{\omega})$, $f$ has only poles of order $1$, at $\widetilde{\o}_1,\ldots,\widetilde{\o}_p$, with residues $r_1,\ldots,r_p$, respectively. Then there exists a constant $c$ such that \eqref{eq:expression_elliptic_zeta} holds.
\item \label{principle_transformation} Let $p$ be some positive integer. The Weierstrass elliptic function with periods $\overline{\omega},\widehat{\omega}/p$ can be written in terms of ${\wp}$ as
     \begin{equation*}
          {\wp}(\omega)+\sum_{1\leq \ell\leq p-1}[{\wp}(\omega+\ell\widehat{\omega}/p)-{\wp}(\ell\widehat{\omega}/p)],\qquad  \forall \omega\in{\bf C}.
     \end{equation*}
\item\label{half_period_translated_zeta}The function $\zeta$ is quasi-periodic, in the sense that
\begin{equation*}
     \zeta(\o+\overline{\o}) = \zeta(\o)+2\zeta(\overline{\o}/2),\qquad  \zeta(\o+\widehat{\o}) = \zeta(\o)+2\zeta(\widehat{\o}/2),\qquad \forall \omega\in{\bf C}.
\end{equation*}
\item\label{Frobenius}If $\alpha+\beta+\gamma=0$ then
\begin{equation*}
     (\zeta(\alpha)+\zeta(\beta)+\zeta(\gamma))^2=\wp(\alpha)+\wp(\beta)+\wp(\gamma).
\end{equation*}
\item\label{bisection}We have the bisection formula:
\begin{align*}
     \wp(\o/2)=\wp(\o)&+\sqrt{(\wp(\o)-\wp(\o_1/2))(\wp(\o)-\wp(\o_2/2))}\\
     &+\sqrt{(\wp(\o)-\wp(\o_1/2))(\wp(\o)-\wp((\o_1+\o_2)/2))}\\
     &+\sqrt{(\wp(\o)-\wp(\o_2/2))(\wp(\o)-\wp((\o_1+\o_2)/2))},\qquad  \forall \omega\in{\bf C}.
\end{align*}
\end{enumerate}
\end{lem}

\section*{Acknowledgments}

We thank the two anonymous referees for their useful and insightful comments and suggestions, which led us to improve the presentation of the paper.

\end{document}